\documentclass[11pt]{amsart}

\usepackage{amsmath,amssymb,latexsym,soul,cite,mathrsfs,amsfonts,amssymb}
\usepackage{subfigure}
\usepackage{float}

\usepackage{color,enumitem,graphicx}
\usepackage[colorlinks=true,urlcolor=blue,
citecolor=red,linkcolor=blue,linktocpage,pdfpagelabels,
bookmarksnumbered,bookmarksopen]{hyperref}
\usepackage[english]{babel}
\usepackage{comment}

\usepackage[left=2.9cm,right=2.9cm,top=2.8cm,bottom=2.8cm]{geometry}
\usepackage[hyperpageref]{backref}

\usepackage[colorinlistoftodos]{todonotes}
\makeatletter
\providecommand\@dotsep{5}
\def\listtodoname{List of Todos}
\def\listoftodos{\@starttoc{tdo}\listtodoname}
\makeatother
\makeatletter
\newcommand{\mylabel}[2]{#2\def\@currentlabel{#2}\label{#1}}
\makeatother

\usepackage{verbatim} % to using the command "\comment"

\numberwithin{equation}{section}
\newtheorem{theorem}{Theorem}[section]
\newtheorem{prop}[theorem]{Proposition}
\newtheorem{lem}[theorem]{Lemma}
\newtheorem{cor}[theorem]{Corollary}
\newtheorem{rem}{Remark}

\newcommand\restr[2]{{% we make the whole thing an ordinary symbol
  \left.\kern-\nulldelimiterspace % automatically resize the bar with \right
  #1 % the function
  \vphantom{\big|} % pretend it's a little taller at normal size
  \right|_{#2} % this is the delimiter
  }}
\title[Abstract Bifurcation Result]
{On an Abstract Bifurcation Result Concerning Homogeneous Potential Operators with Applications to PDEs}

\author[K. Silva]{Kaye Silva}

\address[K. Silva]{\newline\indent
	Instituto de Matem\'atica e Estat\'istica.   
	\newline\indent 
	Universidade Federal de Goi\'as,
	\newline\indent
	74001-970, Goi\^ania, GO, Brazil}
\email{\href{mailto:kayeoliveira@hotmail.com}{kayeoliveira@hotmail.com}, \href{mailto:kaye_0liveira@ufg.br}{kaye\_0liveira@ufg.br}}

\thanks{
The author was  partially supported by CNPq/Brazil under Grant [408604/2018-2]. 
}

\begin{document}

\begin{abstract} We study an abstract equation in a reflexive Banach space, depending on a real parameter $\lambda$. The equation is composed by homogeneous potential operators. By analyzing the Nehari sets, we prove a bifurcation result. In some particular cases we describe the full bifurcation diagram, and in general, we estimate the parameter $\lambda_b$ for which the problem does not have non-zero solution where $\lambda>\lambda_b$. We give many applications to partial differential equations: Kirchhoff type equations,  
Schr\"odinger equations coupled with the electromagnetic field, Chern-Simons-Schr\"odinger systems and a nonlinear eigenvalue problem.
\end{abstract}

\maketitle

\tableofcontents

\section{Introduction}
Let $X$ be a reflexive Banach space and $X^*$ its topological dual. We denote by $\|\cdot \|$ the norm on $X$. Consider the following equation
\begin{equation}\label{p}
%\tag{$*$}
\left\{
\begin{aligned}
P(u)+\lambda T&(u)-Q(u)=0 \\
u&\in X                                
\end{aligned}
\right.,
\end{equation}
where $P,T,Q:X\to X^*$ and $\lambda$ is a real positive parameter. Our proposal in this work is to describe a bifurcation result to equation \eqref{p}. We will consider the following hypothesis (see Chabrowski \cite{chab} for a detailed account of homogeneous potential operators):
\begin{enumerate}
		\item[\mylabel{H}{($H$)}]  $P,T,Q:X\to X^*$ are $p-1$, $\gamma-1$ and $q-1$-homogeneous potential operators respectively, with $1<p<q<\gamma$. 
\end{enumerate}
In particular $P(0)=T(0)=Q(0)=0$ and the functions $\mathcal{P},\mathcal{T},\mathcal{Q}:X\to \mathbb{R}$ defined by
\begin{equation*}
\mathcal{P}(u)=P(u)u,\ \mathcal{T}(u)=T(u)u,\ \mbox{and}\ \mathcal{Q}(u)=Q(u)u,
\end{equation*}
are $C^1$ and 
\begin{equation*}
\frac{1}{p}\mathcal{P}'(u)=P(u),\ \frac{1}{\gamma}\mathcal{T}'(u)=T(u),\ \frac{1}{q}\mathcal{Q}'(u)=Q(u),u\in X.
\end{equation*}
We also assume the following hypothesis: 
\begin{enumerate}	
	\item[\mylabel{C}{($C$)}]  The functions $u\mapsto \mathcal{P}(u)$ and $u\mapsto \mathcal{T}(u)$ are weakly lower semi-continuous on $X$ and $Q$ is strongly continuous.
	\item[\mylabel{E1}{($E_1$)}]  $\mathcal{P}(u)>0$, $\mathcal{T}(u)>0$ and $\mathcal{Q}(u)>0$ for each $u\in X\setminus\{0\}$.
	\item[\mylabel{E2}{($E_2$)}]  There exist constants $C_1,C_2>0$ such that $\mathcal{P}(u)\ge C_1\|u\|^p$ and $\mathcal{Q}(u)\le C_2\|u\|^q$ for each $u\in X$.
	\item[\mylabel{E3}{($E_3$)}]  There exists a constant $C>0$ such that 
	\begin{equation}
	\frac{\mathcal{Q}(u)^{\frac{\gamma-p}{q-p}}}{\mathcal{T}(u)\mathcal{P}(u)^{\frac{\gamma-q}{q-p}}}\le C, \forall u\in X\setminus\{0\}.
	\end{equation}
	\end{enumerate}
Define $\Phi_\lambda:X\to \mathbb{R}$ by
\begin{equation*}
\Phi_\lambda(u)=\frac{1}{p}\mathcal{P}(u)+\frac{\lambda}{\gamma}\mathcal{T}(u)-\frac{1}{q}\mathcal{Q}(u).
\end{equation*}
We say that a solution to equation \eqref{p} is a critical point of $\Phi_\lambda$. In order to find critical points to $\Phi_\lambda$, we need a compactness condition:
\begin{enumerate}
	\item[\mylabel{PS}{($PS$)}] The energy functional $\Phi_\lambda$ satisfies the $(PS)$ condition uniformly in $\lambda>0$, that is, if $\lambda_n\to\lambda> 0$ and $u_n\in X$ satisfies $\Phi_\lambda(u_n)$ is bounded and $\Phi_\lambda'(u_n)\to 0$ as $n\to \infty$, then $u_n$ has a convergent subsequence.
\end{enumerate}
This work is mainly motivated by the recently work of Il'yasov \cite{ilyasENMM} concerning the so-called extreme value for the application of the Nehari manifold method. Indeed, consider an equation of the form
\begin{equation}\label{ileqa}
G'(u)-\lambda F'(u)=0,u\in X,
\end{equation}
where $F,G$ are $C^1$ functions in $X$ and $\Phi_\lambda(u)=G(u)-\lambda F(u)$ is such that $\Phi_\lambda'(u)=G'(u)-\lambda F'(u)$. One can associate to equation \eqref{ileqa} its Nehari set (see Nehari \cite{neh,neh1}) given by 
\begin{equation*}
\mathcal{N}_\lambda=\{u\in X:\Phi_\lambda'(u)u=0 \}.
\end{equation*}
Observe that every critical point of $\Phi_\lambda$ belongs to $\mathcal{N}_\lambda$. Can we say that local minimum points of $\Phi_\lambda$ constrained to $\mathcal{N}_\lambda$ are critical points of $\Phi_\lambda$ on the whole space? In the cited works of Nehari, it turns out that the answer to this question was positive since there the Nehari set was in fact a codimension $1$, $C^1$ manifold,  know as Nehari manifold, however in general this is not true and $\mathcal{N}_\lambda$ does not need to be a $C^1$ manifold. So this leads us to our second question: for what values of the parameter $\lambda>0$ does $\mathcal{N}_\lambda$ is a manifold? 

The extreme values were introduced in \cite{ilyasENMM} and they define thresholds for the applicability of the Nehari manifold method, that is, regions on $\mathbb{R}$ for which $\mathcal{N}_\lambda$ is a manifold.  They are found through the study of the so-called Nonlinear Rayleigh Quotient given by
\begin{equation*}
X\ni u\mapsto \frac{G'(u)u}{F'(u)u}.
\end{equation*}
As we will see in the applications, under our hypotheses, there are cases where the Nehari manifold methods is not applicable since $\mathcal{N}_\lambda$ is not a manifold for any $\lambda>0$, nevertheless we were able to provide the existence of two extreme parameters $0<\lambda_0^*<\lambda^*<\infty$ for which 
\begin{theorem}\label{inexistenceresult} 
	There exists $\varepsilon>0$ such that for each $\lambda\in (0,\lambda_0^*+\varepsilon)$ problem \eqref{p} has two solutions $u_\lambda,w_\lambda\in X\setminus\{0\}$. Moreover if $\lambda\in(0,\lambda_0^*]$, then $u_\lambda$ is a global minimizer to $\Phi_\lambda$, while $w_\lambda$ is a mountain pass solution satisfying
	\begin{equation*}
	\Phi_\lambda(u_\lambda)<0<\Phi_\lambda(w_\lambda), \forall \lambda\in(0,\lambda_0^*),
	\end{equation*}
	and
	\begin{equation*}
	\Phi_{\lambda_0^*}(u_{\lambda_0^*})=0<\Phi_{\lambda_0^*}(w_{\lambda_0^*}).
	\end{equation*}
	If $\lambda\in(\lambda_0^*,\lambda_0^*+\varepsilon)$, then $u_\lambda$ is a local minimizer to $\Phi_\lambda$, while $w_\lambda$ is a mountain pass solution satisfying
	\begin{equation*}
	0<\Phi_\lambda(u_\lambda)<\Phi_\lambda(w_\lambda), \forall \lambda\in(\lambda_0^*,\lambda_0^*+\varepsilon).
	\end{equation*}
Furthermore problem \eqref{p} does not have non-zero solution if $\lambda>\lambda^*$.
\end{theorem}
An important question which follows Theorem \ref{inexistenceresult} concerns the parameter
\begin{equation*}
\lambda_b=\sup\{\lambda>0:\mbox{equation}\ \eqref{p}\ \mbox{has non-zero solutions}\}.
\end{equation*}
It turns out that there are examples where $\lambda_b<\lambda^*$ and $\lambda_b=\lambda^*$. The next result consider a particular case where $\lambda_b=\lambda^*$.
\begin{theorem}\label{intglobal}
Assume that $\mathcal{T}(u)=\mathcal{P}(u)^{\frac{\gamma}{p}},\ \forall u\in X$. For each $\lambda\in (0,\lambda^*)$ problem \eqref{p} has two solutions $u_\lambda,w_\lambda\in X\setminus\{0\}$. Moreover if $\lambda\in(0,\lambda_0^*]$, then $u_\lambda$ is a global minimizer to $\Phi_\lambda$, while $w_\lambda$ is a mountain pass solution satisfying
\begin{equation*}
\Phi_\lambda(u_\lambda)<0<\Phi_\lambda(w_\lambda), \forall \lambda\in(0,\lambda_0^*),
\end{equation*}
and
\begin{equation*}
\Phi_{\lambda_0^*}(u_{\lambda_0^*})=0<\Phi_{\lambda_0^*}(w_{\lambda_0^*}).
\end{equation*}
If $\lambda\in(\lambda_0^*,\lambda^*)$, then $u_\lambda$ is a local minimizer to $\Phi_\lambda$, while $w_\lambda$ is a mountain pass solution satisfying
\begin{equation*}
0<\Phi_\lambda(u_\lambda)<\Phi_\lambda(w_\lambda), \forall \lambda\in(\lambda_0^*,\lambda^*).
\end{equation*}
	Moreover, there exists at least one solution $v_{\lambda^*}\in X$ to equation \eqref{p} with $\lambda=\lambda^*$ and
	\begin{equation*}
	\lim_{\lambda\uparrow \lambda^*}\Phi_\lambda(u_\lambda)=\lim_{\lambda\uparrow \lambda^*}\Phi_\lambda(w_\lambda)=\Phi_{\lambda^*}(v_{\lambda^*})=\frac{\gamma-p}{pq\gamma}\frac{(q-p)^{\frac{\gamma}{\gamma-p}}}{(\gamma-q)^{\frac{p}{\gamma-p}}}\frac{1}{(\lambda^*)^{\frac{p}{\gamma-p}}.}
	\end{equation*}
	Furthermore
	\begin{equation*}
	\lambda_b=\lambda^*.
	\end{equation*}
\end{theorem}
As an application of Theorem \ref{intglobal} let us consider the following Kirchhoff type equation (see Subsection \ref{SUBKIr}) 
\begin{equation}\label{KPI}
%\tag{$*$}
\left\{
\begin{aligned}
-\left(a+\lambda\int |\nabla u|^2\right)\Delta u&= |u|^{q-2}u &&\mbox{in}\ \ \Omega, \\
u&=0                                   &&\mbox{on}\ \ \partial\Omega,
\end{aligned}
\right.
\end{equation}
where $a>0$, $\lambda>0$, $q\in (2,4)$ and $\Omega\subset \mathbb{R}^3$ is a bounded regular domain. We improve the results of Silva \cite{ka} with:
\begin{theorem}\label{KI} For each $\lambda\in(0,\lambda^*)$ problem \ref{KPI} has two positive classical solutions $u_\lambda,w_\lambda$. Moreover, problem \eqref{KPI} has at least a positive classical solutions when $\lambda=\lambda^*$ and does not have non-zero solutions for $\lambda>\lambda^*$, that is, $\lambda_b=\lambda^*$.
\end{theorem}
We also give an example where $\lambda_b<\lambda^*$. Consider the following Schr\"odinger equation coupled with the electromagnetic field in $\mathbb R^{3}$ (see Subsection \ref{SUBSEE}):
\begin{equation}\label{SEI}
\left\{
\begin{aligned}
-&\Delta u+\omega u+\lambda\phi u=|u|^{q-2}u, \\
&-\Delta \phi+a^2\Delta^2 \phi = 4\pi u^2,
\end{aligned}
\right.
\end{equation}
where $a> 0$, $\omega> 0$, $\lambda>0$ and $q\in (2,6)$. As an application of Theorem \ref{inexistenceresult} we have a slight improvement of Siciliano and Silva \cite{gaeka}:
\begin{theorem}\label{INTRSEE} There exists $\varepsilon>0$ such that  problem \ref{SEI} has two pairs of positive classical solutions $(u_\lambda,\phi_\lambda)$ and $(w_\lambda,\tilde{\phi}_\lambda)$ for each $\lambda\in(0,\lambda_0^*+\varepsilon)$. Moreover $\lambda_b<\lambda^*$.
\end{theorem}
We give a third application: consider the nonlinear eigenvalue problem (see Subsection \ref{SUBNEP})
\begin{equation}\label{NEPI}
%\tag{$*$}
\left\{
\begin{aligned}
-\Delta u+\lambda|u|^{\gamma-2}u&= \mu|u|^{q-2}u &&\mbox{in}\ \ \Omega, \\
u&=0                                   &&\mbox{on}\ \ \partial\Omega,
\end{aligned}
\right.
\end{equation}
where $2<q<\gamma<2^*$, $\Omega$ is a bounded regular domain and $\lambda,\mu$ are positive parameters.
\begin{theorem}\label{THNEPI} For each $\mu>0$, there exists $\varepsilon(\mu)>0$ such that for each $\lambda\in(0,\lambda_0^*(\mu)+\varepsilon(\mu))$ problem \eqref{NEPI} has two positive classical solutions $u_\lambda,w_\lambda$. Moreover $\lambda_b(\mu)<\lambda^*(\mu)$.
\end{theorem}
As an byproduct we improve Theorem 2.32 of Rabinowitz \cite{rabino} 
\begin{theorem}\label{rabinoimprovedI} Suppose that $\mu=\lambda$, then there exists $\varepsilon>0$ and $\lambda_*>0$ such that for each $\lambda>\lambda_*-\varepsilon$ problem \eqref{NEPI} admits two positive classical solutions $u_\lambda,w_\lambda$ satisfying:
	\begin{enumerate}
		\item If $\lambda\ge \lambda_*$, then $u_\lambda$ is a global minimizer to $\Phi_\lambda$ while $w_\lambda$ is a mountain pass solution and
		\begin{equation*}
		\Phi_\lambda(u_\lambda)<0<\Phi_\lambda(w_\lambda), \forall\ \lambda\in[\lambda_*,\infty),
		\end{equation*}
		\begin{equation*}
		\Phi_{\lambda_*}(u_{\lambda_*})=0<\Phi_{\lambda_*}(w_{\lambda_*}).
		\end{equation*}
		\item If $\lambda_*-\varepsilon<\lambda<\lambda_*$, then $u_\lambda$ is a local minimizer to $\Phi_\lambda$ while $w_\lambda$ is a mountain pass solution and
		\begin{equation*}
		0<\Phi_\lambda(u_\lambda)<\Phi_\lambda(w_\lambda), \forall\ \lambda\in(\lambda_*-\varepsilon,\lambda_*).
		\end{equation*}
	\end{enumerate} 
	Moreover, if $\mu_0>0$ is the unique value for which 
	\begin{equation*}
	\mu_0=\lambda^*(\mu_0),
	\end{equation*}
	then problem \eqref{NEP} does not have non-zero solutions for $0<\lambda<\mu_0$.
\end{theorem}
The last application concerns the following gauged Schr\"odinger equation in dimension $2$ including the so-called Chern-Simons term (see Subsection \ref{SUBCSS}):
\begin{equation}\label{CSSI}
%\tag{$*$}
\left\{
\begin{aligned}
-\Delta u+u+\lambda\left(\frac{h^2(|x|)}{|x|^2}+\int_{|x|}^\infty\frac{h(s)}{s}u^2(s)ds\right)u&= |u|^{q-2}u &&\mbox{in}\ \ \mathbb{R}^2, \\
u\in H_r^1(\mathbb{R}^2),&                                 
\end{aligned}
\right.
\end{equation}
where $\lambda>0$ is a real positive parameter, $q\in (2,4)$ and 
\begin{equation*}
h(s)=\frac{1}{2}\int _0^s ru^2(r)dr.
\end{equation*}
\begin{theorem}[Xia \cite{xia}]\label{THMCSS} There exists $\varepsilon>0$ such that for each $\lambda\in(0,\lambda_0^*+\varepsilon)$ problem \eqref{CSSI} has two non-negative solutions $u_\lambda,w_\lambda$.
\end{theorem}

Bifurcation problems have many applications and a long history (see for example Crandall and Rabinowitz \cite{crarabino} and the references therein). The method described here does not make use of second derivatives and although we do not provide a full bifurcation picture (only in a particular case we provide it), it relies on simple analysis as to the use of the fibering method of Pohozaev \cite{poh}, combined with the Nehari sets and standard minimization and min-max arguments. Similar ideas have been employed, for example, in convex-concave problems (see Brown and Wu \cite{browu}). 

This work is organized as follows: In the next Section we prove existence of solutions to equation \eqref{p} for $\lambda\in(0,\lambda_0^*]$. In Section \ref{S3} we consider the problem of non-existence of solutions. In Section \ref{S4} we deal with the existence of solutions when $\lambda_0^*<\lambda<\lambda^*$ and prove Theorem \ref{inexistenceresult}. In the first Subsection of Section \ref{S5} we provide some technical results which are used to understand the value $\lambda_b$. In the second Subsection we consider a particular example where $\lambda_b=\lambda^*$ and prove Theorem \ref{intglobal}. In Section \ref{S6} we give the applications (Theorems \ref{KI}, \ref{INTRSEE}, \ref{THNEPI}, \ref{rabinoimprovedI}, \ref{THMCSS}).
\section{Existence Of Two Solutions When $\lambda\in(0,\lambda_0^*]$}
In this Section we prove the following theorem
\begin{theorem}\label{existence}For each $\lambda\in (0,\lambda_0^*]$ the problem \eqref{p} has two solutions $u_\lambda,w_\lambda\in X\setminus\{0\}$. Moreover $u_\lambda$ is a global minimizer while $w_\lambda$ is a mountain pass solution satisfying
	\begin{equation*}
	\Phi_\lambda(u_\lambda)<0<\Phi_\lambda(w_\lambda), \forall\ \lambda\in(0,\lambda_0^*),
	\end{equation*}
	and
	\begin{equation*}
	\Phi_{\lambda_0^*}(u_{\lambda_0^*})=0<\Phi_{\lambda_0^*}(w_{\lambda_0^*}).
	\end{equation*}
\end{theorem}
In order to prove Theorem \ref{existence} we need some preliminary results. For $\lambda>0$ define
\begin{equation*}
\hat{\Phi}_\lambda=\inf\{\Phi_\lambda(u):u\in X\}.
\end{equation*}
\begin{prop}\label{globalm} Let $\lambda>0$ and suppose that there exists $u\in X\setminus\{0\}$ such that $\Phi_\lambda(u)<0$. Then there exists $u_\lambda\in X\setminus\{0\}$ such that
	\begin{equation*}
	\hat{\Phi}_\lambda=\Phi_\lambda(u_\lambda)\ \ \ \mbox{and}\ \ \ \Phi'_\lambda(u_\lambda)=0.
	\end{equation*} 
\end{prop}
\begin{proof} Indeed, note that $\hat{\Phi}_\lambda<0$. From \ref{C} we have that $\Phi_\lambda$ is weakly lower semi-continuous. Once \ref{PS} is satisfied we can use the Ekeland's Variational Principle to find $u_\lambda\in X\setminus\{0\}$ such that $\hat{\Phi}_\lambda=\Phi_\lambda(u_\lambda)$ and once $\Phi_\lambda$ is $C^1$ we also have that $\Phi'_\lambda(u_\lambda)=0$.
\end{proof}
Since Proposition \ref{globalm} implies the existence of the global minimum of Theorem \ref{existence}, we need to verify for what values of $\lambda>0$ does there exists $u\in X\setminus\{0\}$ such that $\Phi_\lambda(u)<0$. To this end we study the Nehari sets associated with the energy functional $\Phi_\lambda$.

For each $\lambda>0$ and $u\in X\setminus\{0\}$, consider the fiber map $\varphi_{\lambda,u}:(0,\infty)\to \mathbb{R}$ defined by
\begin{equation*}
\varphi_{\lambda,u}(t)=\Phi_\lambda(tu).
\end{equation*}
We introduce the Nehari set
\begin{equation*}
\mathcal{N}_\lambda=\{u\in X:\ \varphi'_{\lambda,u}(1)=0 \},
\end{equation*}
and note that
\begin{equation*}
\mathcal{N}_\lambda=\mathcal{N}_\lambda^+\cup \mathcal{N}_\lambda^0\cup \mathcal{N}_\lambda^-,
\end{equation*}
where 
\begin{equation*}
\mathcal{N}_\lambda^+=\{u\in H_0^1(\Omega)\setminus\{0\}:\ \varphi'_{\lambda,u}(1)=0,\ \varphi''_{\lambda,u}(1)>0 \},
\end{equation*}
\begin{equation*}
\mathcal{N}_\lambda^0=\{u\in H_0^1(\Omega)\setminus\{0\}:\ \varphi'_{\lambda,u}(1)=0,\ \varphi''_{\lambda,u}(1)=0 \},
\end{equation*}
and
\begin{equation*}
\mathcal{N}_\lambda^-=\{u\in H_0^1(\Omega)\setminus\{0\}:\ \varphi'_{\lambda,u}(1)=0,\ \varphi''_{\lambda,u}(1)<0 \}.
\end{equation*}
The next result, which is just an application of the implicit function theorem, shows that $\mathcal{N}_\lambda^+, \mathcal{N}_\lambda^-$ are $C^1$ manifolds know as Nehari manifolds.
\begin{lem}\label{constrained} If $\mathcal{N}_\lambda^+$, $\mathcal{N}_\lambda^-$ are non empty, then they are $C^1$ manifolds of codimension $1$ in $X$. Moreover, if $u\in \mathcal{N}_\lambda^+\cup \mathcal{N}_\lambda^-$ is a critical point to $\Phi_\lambda$ restricted to $\mathcal{N}_\lambda^+\cup \mathcal{N}_\lambda^-$, then $u$ is a critical point to $\Phi_\lambda$.
\end{lem}
The following result will prove useful
\begin{lem}\label{bound} For each $\lambda>0$ for which $\mathcal{N}_\lambda\neq\emptyset$ there holds 
	\begin{equation*}
	\|u\|\ge\left(\frac{C_1}{C_2}\right)^{\frac{1}{q-p}},\ \forall u\in \mathcal{N}_\lambda.
	\end{equation*}
\end{lem}
\begin{proof} Indeed, if $u\in \mathcal{N}_\lambda$, then
	\begin{equation*}
	\mathcal{P}(u)+\lambda\mathcal{T}(u)=\mathcal{Q}(u).
	\end{equation*}
	From \ref{E1} and \ref{E2} we conclude that 
	\begin{equation*}
	C_1\|u\|^p\le C_2\|u\|^q,
	\end{equation*}
	and since $p<q$ we obtain 
	\begin{equation*}
	\|u\|\ge \left(\frac{C_1}{C_2}\right)^{\frac{1}{q-p}}.
	\end{equation*}
\end{proof}
Now we study the fiber maps. Due to hypothesis \ref{H} we have the following
\begin{prop}\label{fibering}For each $\lambda>0$ and $u\in X\setminus\{0\}$, there are only three possibilities for the graph of $ \varphi_{\lambda,u}$
	\begin{enumerate}
		\item[I)] The function $\varphi_{\lambda,u}$ has only two critical points, to wit, $0<t_\lambda^-(u)<t_\lambda^+(u)$. Moreover, $t_\lambda^-(u)$ is a local maximum with $\varphi''_{\lambda,u}(t_\lambda^-(u))<0$ and $t_\lambda^+(u)$ is a local minimum with $\varphi''_{\lambda,u}(t_\lambda^+(u))>0$;   
		\item[II)] The function  $\varphi_{\lambda,u}$ has only one critical point when $t>0$ at the value $t_\lambda(u)$. Moreover, $\varphi''_{\lambda,u}(t_\lambda(u))=0$ and  $\varphi_{\lambda,u}$ is increasing;
		\item[III)] The function  $\varphi_{\lambda,u}$ is increasing and does not have critical points.
	\end{enumerate}
\end{prop}
Observe from Proposition \ref{fibering} that the existence of $u\in X\setminus\{0\}$ satisfying $\Phi_\lambda(u)<0$ is possible if, and only if, $\varphi_{\lambda,u}(t_\lambda^+(u))<0$, which leads us to study the following system of equations: for $u\in X\setminus\{0\}$ consider the system
\begin{equation}\label{zeroenergy}
\left\{
\begin{aligned}
\varphi_{\lambda,u}(t) &= 0 \\
\varphi'_{\lambda,u}(t) &= 0
\end{aligned}
\right.,
\end{equation}
which has a unique solution $(t_0(u),\lambda_0(u))$  given by
\begin{equation*}
t_0(u)=\left(\frac{1}{\lambda}\frac{\gamma}{q}\frac{q-p}{\gamma-p}\frac{\mathcal{Q}(u)}{\mathcal{T}(u)}\right)^{\frac{1}{\gamma-q}}.
\end{equation*}
\begin{equation*}
\lambda_0(u)=\frac{\gamma}{q}\frac{q-p}{\gamma-p}\left(\frac{p}{q}\frac{\gamma-q}{\gamma-p}\right)^{\frac{\gamma-q}{q-p}}\frac{\mathcal{Q}(u)^{\frac{\gamma-p}{q-p}}}{\mathcal{T}(u)\mathcal{P}(u)^{\frac{\gamma-q}{q-p}}}.
\end{equation*}
The functions $\lambda_0(u)$ has the following geometrical interpretation:
\begin{prop}\label{fiberingvariation0} If $u\in H_0^1(\Omega)\setminus\{0\}$, then $\lambda_0(u)$ is the unique parameter $\lambda>0$ for which the fiber map $\varphi_{\lambda,u}$ has a global minimum critical point with zero energy at $t_0(u)$. Moreover, if $0<\lambda<\lambda_0(u)$, then $\inf_{t>0}\varphi_{\lambda,u}(t)<0$ while if $\lambda>\lambda(u)$, then $\inf_{t>0}\varphi_{\lambda,u}(t)=0$.
\end{prop}
\begin{proof} The uniqueness of $\lambda_0(u)$ comes from equation \eqref{zeroenergy}. If $0<\lambda<\lambda_0(u)$ then from the definition we have
	\begin{equation*}
	\varphi_{\lambda,u}(t_0(u))<\varphi_{\lambda_0(u),u}(t_0(u))=0,
	\end{equation*}
	which implies that $\inf_{t>0}\varphi_{\lambda,u}(t)<0$. If $\lambda>\lambda_0(u)$ then
	\begin{equation*}
	\varphi_{\lambda,u}(t))>\varphi_{\lambda_0(u),u}(t)\ge0, \forall t>0,
	\end{equation*}
	and therefore $\inf_{t>0}\varphi_{\lambda,u}(t)=\varphi_{\lambda,u}(0)=0$.
	\end{proof}
\begin{prop}\label{extremal0}The function $u\mapsto \lambda_0(u)$, $u\in X\setminus\{0\}$ is $0$-homogeneous, continuous and bounded from above.
\end{prop}
\begin{proof}The $0$-homogeneity and continuity are straightforward and the boundedness follows from \ref{E3}. 
	
\end{proof}
Define
\begin{equation*}
\lambda_0^*=\sup\{\lambda_0(u):\ u\in X\setminus\{0\}\},
\end{equation*}
and observe from Proposition \ref{extremal0} that $\lambda_0^*<\infty$. As a straightforward consequence of Propositions \ref{fiberingvariation0} and \ref{extremal0} we have an answer to the question ``for what values of $\lambda>0$ does there exists $u\in X\setminus\{0\}$ such that $\Phi_\lambda(u)<0$?"
\begin{cor}\label{existecem} There exists $u\in X$ such that $\Phi_\lambda(u)<0$ if, and only if $\lambda\in (0,\lambda_0^*)$.
\end{cor}
Now we turn our attention to the second solution. If $\lambda\in(0,\lambda_0^*)$, then from Corollary \ref{existecem} there exists $v_\lambda$ such that $\Phi_\lambda(v_\lambda)<0$. Define
\begin{equation*}
c_\lambda=\inf_{u\in \Gamma_\lambda}\max_{t\in[0,1]}\Phi_\lambda,
\end{equation*}
where
\begin{equation*}
\Gamma_\lambda=\{\psi\in C([0,1],X):\ \psi(0)=0,\ \psi(1)=v_\lambda\}.
\end{equation*}
In order to provide a Mountain Pass Geometry to the function $\Phi_\lambda$ we prove the following
\begin{prop}\label{MPG} There exist $C_\lambda>0$ and $\rho_\lambda>0$ satisfying
	\begin{equation*}
	\Phi_\lambda(u)\ge C_\lambda,\ \forall u\in H_0^1(\Omega),\ \|u\|=\rho_\lambda,
	\end{equation*}
	and
	\begin{equation*}
	\lim_{C_\lambda\to 0}\rho_\lambda=0.
	\end{equation*}
\end{prop}
\begin{proof} Indeed from \ref{E1} and \ref{E2} we have the inequality
	\begin{equation*}
	\Phi_\lambda(u)\ge C_1\|u\|^p-C_2\|u\|^q,\ \forall\ u\in X,
	\end{equation*}
	and since $p<q$, the proof is complete.	
\end{proof}
The existence of a mountain pass critical point is immediately:
\begin{cor}\label{MPS} For each $\lambda\in (0,\lambda_0^*)$, there exists $w_\lambda\in X\setminus\{0\}$ such that $\Phi_\lambda(w_\lambda)=c_\lambda$ and $\Phi'_\lambda(w_\lambda)=0$.
\end{cor}
\begin{proof} The Mountain Pass Geometry given by Proposition \ref{MPG} combined with \ref{PS} implies the existence of  $w_\lambda\in X\setminus\{0\}$ such that $\Phi_\lambda'(w_\lambda)=0$.
	
\end{proof}
\begin{proof}[Proof of Theorem \ref{existence}] Suppose that $\lambda\in(0,\lambda_0^*)$, then Corollary \ref{existecem} combined with Proposition \ref{globalm} implies the existence of a global minimizer $u_\lambda$ such that $\Phi_\lambda(u_\lambda)<0$. The mountain pass critical point comes from Corollary \ref{MPS}. The case $\lambda=\lambda_0^*$ goes as following: 
	\vskip.3cm
	$1)$: Global Minimizer
	\vskip.3cm
	
	Take a sequence $\lambda_n\uparrow \lambda_0^*$ and a corresponding sequence $u_n:=u_{\lambda_n}$ satisfying $\Phi_{\lambda_n}(u_n)<0$ and $\Phi'_{\lambda_n}(u_n)=0$ for each $n\in\mathbb{N}$. Since $0<\lambda<\lambda'<\lambda_0^*$ implies that $\Phi_{\lambda}(u)<\Phi_{\lambda'}(u)$ for each $u\in X\setminus\{0\}$ we conclude that $\hat{\Phi}_\lambda\le\Phi_\lambda(u_{\lambda'})<\Phi_{\lambda'}({u_{\lambda'}})=\hat{\Phi}_{\lambda'}$ and hence we can assume without loss of generality that 
	\begin{equation*}
	\lim_{n\to \infty}\Phi_{\lambda_n}(u_n)\to c\le 0\ \ \ \mbox{and}\ \ \ \Phi'_{\lambda_n}(u_n)=0,\ n\in \mathbb{R}.
	\end{equation*}
	From \ref{PS} we conclude that $u_n \to u$ in $X$ and from Proposition \ref{bound} it follows that $u\neq 0$. Therefore $\Phi_{\lambda_0^*}(u)=c\le 0$ and  $\Phi'_{\lambda_0^*}(u)=0$. From Proposition \ref{fiberingvariation0} and the definition of $\lambda_0^*$ we conclude that $c=0$ and by setting $u_{\lambda_0^*}$, the proof is complete.
	\vskip.3cm
	$2)$: Mountain Pass Solution
	\vskip.3cm
	Define 
	\begin{equation*}
	c_{\lambda_0^*}=\inf_{u\in \Gamma_\lambda}\max_{t\in[0,1]}\Phi_\lambda,
	\end{equation*}
	where
	\begin{equation*}
	\Gamma_{\lambda_0^*}=\{\psi\in C([0,1],X):\ \psi(0)=0,\ \psi(1)=u_{\lambda_0^*}\}.
	\end{equation*}
	Since $\Phi_{\lambda_0^*}(u_{\lambda_0^*})=0$, Proposition \ref{MPG} combined with Proposition \ref{fibering} and \ref{PS} implies the existence of  $w_{\lambda_0^*}\in X\setminus\{0\}$ such that $\Phi_{\lambda_0^*}(w_{\lambda_0^*})=c_{\lambda_0^*}$ and $\Phi'_{\lambda_0^*}(w_{\lambda_0^*})=0$.
	
	To conclude observe from the definitions that
	\begin{equation*}
	\Phi_\lambda(u_\lambda)=\hat{\Phi}_\lambda<0<c_\lambda=\Phi_\lambda(w_\lambda), \forall\ \lambda\in(0,\lambda_0^*),
	\end{equation*}
	and
	\begin{equation*}
	\Phi_{\lambda_0^*}(u_{\lambda_0^*})=0<c_{\lambda_0^*}=\Phi_{\lambda_0^*}(w_{\lambda_0^*}).
	\end{equation*}  
	\end{proof}
\begin{rem}\label{extremalsol} Since $\Phi_{\lambda_0^*}(u_{\lambda_0^*})=0$ and  $\Phi'_{\lambda_0^*}(u_{\lambda_0^*})=0$ we conclude from the definition of $\lambda_0^*$ that $\lambda_0(u_{\lambda_0^*})=\lambda_0^*$.
\end{rem}
\section{A Non Existence Result}\label{S3}
In this Section we describe a non existence result. To this end observe that if $\Phi'_\lambda(u)=0$, then $u\in \mathcal{N}_\lambda$ therefore, for each $\lambda>0$ for which $\mathcal{N}_\lambda=\emptyset$ we must conclude that problem \eqref{p} does not have non-zero solutions. We characterize the set of $\lambda$ for which $\mathcal{N}_\lambda=\emptyset$ which leads us to study the following system: note that $tu\in \mathcal{N}_\lambda^0$ for $t>0$ and $u\in X\setminus\{0\}$ if and only if 
\begin{equation}\label{extremal}
\left\{
\begin{aligned}
\varphi'_{\lambda,u}(t) &= 0 \\
\varphi''_{\lambda,u}(t) &= 0
\end{aligned}
\right..
\end{equation}
Similar to the system \eqref{zeroenergy}, this system has a unique solution $(t(u),\lambda(u))$ which is given by
\begin{equation}\label{tu}
t(u)=\left(\frac{1}{\lambda}\frac{q-p}{\gamma-p}\frac{\mathcal{Q}(u)}{\mathcal{T}(u)}\right)^{\frac{1}{\gamma-q}}.
\end{equation}
\begin{equation*}
\lambda(u)=\frac{q-p}{\gamma-p}\left(\frac{\gamma-q}{\gamma-p}\right)^{\frac{\gamma-q}{q-p}}\frac{\mathcal{Q}(u)^{\frac{\gamma-p}{q-p}}}{\mathcal{T}(u)\mathcal{P}(u)^{\frac{\gamma-q}{q-p}}}.
\end{equation*}
The function $\lambda(u)$ has the following geometrical interpretation
\begin{prop}\label{fiberingvariation} For each $u\in H_0^1(\Omega)\setminus\{0\}$ we have that $\lambda(u)$ is the unique parameter $\lambda>0$ for which the fiber map $\varphi_{\lambda,u}$ has a critical point with second derivative zero at $t(u)$. Moreover, if $0<\lambda<\lambda(u)$, then $\varphi_{\lambda,u}$ satisfies I) of Proposition \ref{fibering} while if $\lambda>\lambda(u)$, then $\varphi_{\lambda,u}$ satisfies III) of Proposition \ref{fibering}.
\end{prop}
\begin{proof} The uniqueness of $\lambda(u)$ comes from equation \eqref{extremal}. Assume that $\lambda\in(0,\lambda(u))$, then $\varphi_{\lambda,u}$ must satisfies $I)$ or $III)$ of Proposition \ref{fibering}. We claim that it must satisfies $I)$. Indeed, suppose on the contrary that it satisfies $III)$. Once
	\begin{equation*}
	\varphi'_{\lambda(u),u}(t)>\varphi'_{\lambda,u}(t)>0, \forall\ t>0,
	\end{equation*}
	we reach a contradiction since $\varphi'_{\lambda(u),u}(t(u))=0$, therefore $\varphi_{\lambda,u}$ must satisfies $I)$. Now suppose that $\lambda>\lambda(u)$, then 
	\begin{equation*}
	\varphi'_{\lambda,u}(t)>\varphi'_{\lambda(u),u}(t)\ge0, \forall\ t>0,
	\end{equation*}
	and hence $\varphi_{\lambda,u}$ must satisfies $III)$. 
\end{proof}
\begin{prop}\label{extremalvalues} There holds:
	\begin{enumerate}
		\item[i)] 	\begin{equation*}
		\lambda(u)=\frac{q}{\gamma}\left(\frac{q}{p}\right)^{\frac{\gamma-q}{q-p}}\lambda_0(u)\ \ \mbox{and}\  \ \frac{q}{\gamma}\left(\frac{q}{p}\right)^{\frac{\gamma-q}{q-p}}>1.
		\end{equation*} 
		\item[ii)] The function $u\mapsto \lambda(u)$, $u\in X\setminus\{0\}$ is $0$-homogeneous, continuous and bounded from above.
		\item[iii)] There exists $u\in X\setminus\{0\}$ such that
		\begin{equation*}
		\lambda(u)=\sup\{\lambda(v):v\in X\setminus\{0\}\}\ \ \ \mbox{and}\ \ \ \lambda_0(u)=\sup\{\lambda_0(v):v\in X\setminus\{0\}\}.
		\end{equation*}
	\end{enumerate}
\end{prop}
\begin{proof}$i)$ $	\lambda(u)=\frac{q}{\gamma}\left(\frac{q}{p}\right)^{\frac{\gamma-q}{q-p}}\lambda_0(u)$ is obvious. To prove that 
	\begin{equation*}
	C(p,q,\gamma):=\frac{q}{\gamma}\left(\frac{q}{p}\right)^{\frac{\gamma-q}{q-p}}>1,
	\end{equation*}
	just observe that the function $[q,\infty)\ni\gamma\mapsto C(p,q,\gamma)$ is increasing for $1<p<q$ and satisfies $C(p,q,q)=1$.
	
	$ii)$ The $0$-homogeneity and continuity are straightforward and the boundedness follows from \ref{E3}. 
	
	$iii)$ This is a consequence of $i)$ and Remark \ref{extremalsol}.
\end{proof}
From Proposition \ref{extremalvalues} we have that 
\begin{equation*}
\lambda^*:=\sup\{\lambda(v):v\in X\setminus\{0\}\}<\infty,
\end{equation*}
therefore
\begin{theorem}\label{nonex} For each $\lambda>\lambda^*$ problem \eqref{p} does not have non-zero solutions.
\end{theorem}
\begin{proof} In fact, if $\lambda>\lambda^*$ then from Proposition \ref{fiberingvariation} we obtain that $\mathcal{N}_\lambda=\emptyset$ which implies the desired non existence.
\end{proof}
\section{Existence Of Two Solutions Locally Near $\lambda>\lambda_0^*$}\label{S4}
In this Section we analyze the existence of solutions when $\lambda>\lambda_0^*$. Note from Proposition \ref{extremalvalues} that 
\begin{equation*}
\lambda_0^*<\lambda^*,
\end{equation*}	
so it remains to understand what happens on the interval $(\lambda_0^*,\lambda^*]$. We prove the following local result
\begin{theorem}\label{existence1} There exists $\varepsilon>0$ such that for each $\lambda\in (\lambda_0^*,\lambda_0^*+\varepsilon)$ problem \eqref{p} has two solutions $u_\lambda,w_\lambda\in X\setminus\{0\}$. Moreover $u_\lambda$ is a local minimizer while $w_\lambda$ is a mountain pass solution satisfying
	\begin{equation*}
	0<\Phi_\lambda(u_\lambda)<\Phi_\lambda(w_\lambda), \forall \lambda\in(\lambda_0^*,\lambda_0^*+\varepsilon).
	\end{equation*}
\end{theorem}

For $\lambda\in(0,\lambda^*]$, define
\begin{equation*}%\label{MINN}
\hat{J}_\lambda:=\inf \left\{\Phi_\lambda(u):\ u\in \mathcal{N}_\lambda^+\cup \mathcal{N}_\lambda^0 \right\}.
\end{equation*}
Observe that 
\begin{equation}\label{eq:infimo}
\hat{J}_\lambda=\hat{\Phi}_\lambda, \forall \lambda\in (0,\lambda_0^*],
\end{equation}
and from  Proposition \ref{fiberingvariation0} there holds  $\hat{J}_\lambda\ge 0$ for $\lambda \in (\lambda_0^*,\lambda^*]$.
\begin{prop}\label{enerynear} Given $\delta>0$, there exists $\varepsilon>0$ such that for each $\lambda\in (\lambda_0^*,\lambda_0^*+\varepsilon)$ there holds $\hat{J}_\lambda<\delta$.
\end{prop}
\begin{proof} Indeed, let $u_{\lambda_0^*}\in \mathcal {N}_{\lambda_0^*}^+$ be given as in Theorem \ref{existence}. Observe that if $\lambda\downarrow \lambda_0^*$, then $\Phi_\lambda(u_{\lambda_0^*})\to \Phi_{\lambda_0^*}(u_{\lambda_0^*})=0$. Moreover, once $\lambda_0^*<\lambda^*=\lambda(u_{\lambda_0^*})$, it follows that there exists $\varepsilon_1>0$ such that $\lambda_0^*+\varepsilon_1<\lambda(u_{\lambda_0^*})$. From  Propositions \ref{fibering} and \ref{fiberingvariation}, for each $\lambda\in (\lambda_0^*,\lambda_0^*+\varepsilon_1)$, there exists $t_\lambda^+(u_{\lambda_0^*})$  such that $t_\lambda^+(u_{\lambda_0^*})u_{\lambda_0^*}\in \mathcal{N}_\lambda^+$. Note that $t_\lambda^+(u_{\lambda_0^*})\to 1$ as $\lambda\downarrow \lambda_0^*$ and therefore
	\begin{equation*}
	\Phi_\lambda(u_{\lambda_0^*})\le \Phi_\lambda(t_\lambda^+(u_{\lambda_0^*})u_{\lambda_0^*})\to \Phi_{\lambda_0^*}(u_{\lambda_0^*})=0,\ \lambda\downarrow \lambda_0^*.
	\end{equation*}
	If $\varepsilon_2>0$ is choosen in such a way that $\Phi_\lambda(t_\lambda^+(u_{\lambda_0^*})u_{\lambda_0^*})<\delta$ for each $\lambda\in (\lambda_0^*,\lambda_0^*+\varepsilon_2)$, then we set $\varepsilon=\min\{\varepsilon_1,\varepsilon_2\}$ and the proof is complete.
\end{proof}
Before we go further, we need to establish some notation, but first we need the following
\begin{lem}\label{boundN0}If $\mathcal{N}_\lambda^0\neq\emptyset$ then
	\begin{equation*}
	\Phi_\lambda(u)\ge\frac{(\gamma-p)(q-p)}{pq\gamma}\left(\frac{C_1}{C_2}\right)^{\frac{1}{q-p}},\ \forall u\in \mathcal{N}_\lambda^0.
	\end{equation*}
\end{lem}
\begin{proof} Indeed, since $u\in \mathcal{N}_\lambda^0$ then $\varphi'_{\lambda,u}(1)=\varphi''_{\lambda,u}(1)=0$ and hence
	\begin{equation*}
	\Phi_\lambda(u)=\frac{(\gamma-p)(q-p)}{pq\gamma}\mathcal{P}(u),\ \forall u\in \mathcal{N}_\lambda^0.
	\end{equation*}
	From Proposition \ref{bound} the proof is complete.
\end{proof}
Let us recall that by Proposition \ref{MPG}, for each $\lambda>0$ there exist positive constants $\rho_\lambda,C_\lambda$ such that $\Phi_\lambda(u)\ge C_\lambda$ for each $\|u\|=\rho_\lambda$. Since 
	\begin{equation*}
\lim_{C_\lambda\to 0}\rho_\lambda=0,
\end{equation*}
we can assume without loss of generality that (the constant on the right side is given by Lemma \ref{bound})
\begin{equation*}
\rho_\lambda<\left(\frac{C_1}{C_2}\right)^{\frac{1}{q-p}}:=\tilde{C},\ \forall \lambda\in(0,\lambda^*].
\end{equation*} 
We choose $\delta>0$  in Proposition \ref{enerynear} in such a way that 
\begin{equation}\label{eq:delta}
\delta<\min\{C_\lambda,\tilde{D}\},
\end{equation}
where
\begin{equation*}
\tilde{D}:=\frac{(\gamma-p)(q-p)}{pq\gamma}\left(\frac{C_1}{C_2}\right)^{\frac{1}{q-p}},
\end{equation*} 
is the constant given by Lemma \ref{boundN0}. 

From now on we suppose that $\varepsilon>0$ is given as in Proposition \ref{enerynear} in correspondence with the above fixed
$\delta>0$.
\begin{prop}\label{EKE1} There holds
	\begin{equation*}
	\inf\left\{\Phi_\lambda(u): \|u\|\ge \rho_\lambda\right\}=\hat{J}_\lambda, \forall\, \lambda\in (\lambda_0^*,\lambda_0^*+\varepsilon).
	\end{equation*}
\end{prop}
\begin{proof} Indeed, fix $u\in X\setminus\{0\}$ such that $\|u\|=\rho_\lambda$. If by one hand $\varphi_{\lambda,u}$ satisfies $I)$ of Proposition \ref{fibering}, then  
	\begin{equation}\label{E70}
	\inf \{\varphi_{\lambda,u}(t):t\ge 1\}=
	%\tag{$*$}
	\left\{
	\begin{aligned}
	\varphi_{\lambda,u}(1)&\ge C_\lambda &&\mbox{if}\ \ \varphi_{\lambda,u}(1)< \varphi_{\lambda,u}(t_\lambda^+(u)) \\
	\varphi_{\lambda,u}(t_\lambda^+(u))&\ge \hat{J}_\lambda &&\mbox{if}\ \ \varphi_{\lambda,u}(1)\ge \varphi_{\lambda,u}(t_\lambda^+(u))
	\end{aligned}
	\right..
	\end{equation}
	Since $C_\lambda>\delta>\hat{J}_\lambda$, we conclude from \eqref{E70} that 
	\begin{equation}\label{E11}
 \inf \{\varphi_{\lambda,u}(t):t\ge 1\}\ge \hat{J}_\lambda
	\end{equation}
	If on the other hand, $\varphi_{\lambda,u}$ satisfies $II)$ or $III)$ of Proposition \ref{fibering} then
	\begin{equation}\label{E22}
	\inf\{\varphi_{\lambda,u}(t):t\ge 1\}=	\varphi_{\lambda,u}(1)\ge C_\lambda>\delta>\hat{J}_\lambda.
	\end{equation}
	By combining \eqref{E11} with \eqref{E22} and the definition of $\hat{J}_\lambda$, we obtain the desired equality.
\end{proof}
\begin{prop}\label{LocM} For each $\lambda\in (\lambda_0^*,\lambda_0^*+\varepsilon)$ there exists $u_\lambda\in \mathcal{N}_\lambda^+$ such that $\Phi_\lambda(u_\lambda)=\hat{J}_\lambda$. Moreover $\Phi_\lambda(u_\lambda)>0$ and  $\|u_\lambda\|\ge \widetilde C>\rho_\lambda$.
\end{prop}
\begin{proof} Fix $\lambda\in (\lambda_0^*,\lambda_0^*+\varepsilon)$ and let $\{u_n\}\subset \mathcal N_\lambda^{+}\cup \mathcal N_\lambda^{0}$ be a minimizing sequence for $\hat{J}_\lambda<\delta$ by Proposition \ref{enerynear}.  Since $\delta<\min\{C_\lambda,\tilde{D}\}$  and, by Lemma \ref{boundN0}, $\Phi_\lambda(u)\geq \tilde{D}$ on $\mathcal N_\lambda^{0}$, we can assume that $\{u_{n}\}\subset \mathcal N_{q}^{+}$ is bounded away from $\mathcal{N}_\lambda^0$ and hence, by the Ekeland's Variational Principle, we can also suppose that $\Phi'_\lambda(u_n)\to 0$.
	
	We conclude from \ref{PS}  that $u_n\to u$ in $X$ with $\|u\|\ge \widetilde C> \rho_\lambda$. 
	Setting $u_\lambda:= u$ clearly we obtain that  $u_\lambda\in \mathcal{N}_\lambda^+$ and $\Phi_\lambda(u_\lambda)=\hat{J}_\lambda$. Due to the definition of $\lambda_0^*$ and the fact that $\lambda>\lambda_0^*$, we conclude that $\Phi_\lambda(u_\lambda)>0$.
\end{proof}

Now we turn our attention to the second solution. Let $\lambda\in (\lambda_0^*,\lambda_0^*+\varepsilon)$ and $u_\lambda\in \mathcal{N}_\lambda^+$ (given by Proposition \ref{LocM}) such that $\Phi_\lambda(u_\lambda)=\hat{J}_\lambda$. Define
\begin{equation*}%\label{MPA}
c_\lambda=\inf_{\psi\in \Gamma_\lambda}\max_{t\in [0,1]}\Phi_\lambda(\psi(t)),
\end{equation*}
where $\Gamma_\lambda=\{\psi\in C([0,1],X): \psi(0)=0,\ \psi(1)=u_\lambda \}$.
\begin{prop}\label{MPCP} For each $\lambda\in (\lambda_0^*,\lambda_0^*+\varepsilon)$ there exists $w_\lambda\in X\setminus\{0\}$ such that $\Phi_\lambda(w_\lambda)=c_\lambda$ and $\Phi'_\lambda(w_\lambda)=0$. In particular $\Phi_\lambda(w_{\lambda})>\Phi_\lambda(u_\lambda)>0$.
\end{prop}
\begin{proof} Indeed, we combine  Proposition \ref{MPG} with the inequality $\min\{C_\lambda,\tilde{D}\}>\delta\ge \hat{J}_\lambda =\Phi_\lambda(u_\lambda)$ (see \eqref{eq:delta}) and $\|u_\lambda\|>\rho_\lambda$ (see Proposition \ref{LocM}), to obtain a Mountain Pass Geometry for the functional $\Phi_\lambda$. The existence of  $w_\lambda\in X\setminus\{0\}$ satisfying $\Phi_\lambda(w_\lambda)=c_\lambda$ and $\Phi'_\lambda(w_\lambda)=0$ follows from \ref{PS} while the inequality is a consequence of
	\begin{equation*}
	0<\Phi_\lambda(u_\lambda)=\hat{J}_\lambda<\delta<C_\lambda\le c_\lambda=\Phi_\lambda(w_{\lambda}).
	\end{equation*}
\end{proof}
\begin{proof}[Proof of Theorem \ref{existence1}] The local minimizer comes from Proposition \ref{LocM}. In fact, there exists $u_\lambda\in \mathcal{N}_\lambda^+$ such that $\Phi_\lambda(u_\lambda)=\hat{J}_\lambda$ and from Lemma \ref{constrained} it follows that  $\Phi'_\lambda(u_\lambda)=0$. The mountain pass critical point is obtained by Proposition \ref{MPCP}, together with the inequality  $\Phi_\lambda(w_{\lambda})>\Phi_\lambda(u_\lambda)>0$.
\end{proof}	
Now we are in position to prove Theorem \ref{inexistenceresult}:
\begin{proof}[Proof of Theorem \ref{inexistenceresult}] Indeed, the proof is a consequence of Theorems \ref{existence}, \ref{nonex} and \ref{existence1}.

\end{proof}
\section{Global Existence of Solutions And The Turning Point}\label{S5}
In this Section we study the maximal parameter for which problem \eqref{p} has non zero solutions. We divite it in two Subsections. In the first one we give some abstract results which will be used to study globally existence of solutions for $\lambda\in(0,\lambda^*]$. In the second Subsection we consider a particular case of equation \eqref{p} for which existence of solution is provided for all $\lambda\in(0,\lambda^*]$. In the next Section we give examples where \eqref{p} does not have solutions for $\lambda$ close to $\lambda^*$.

\subsection{General Results} Define
\begin{equation*}
\lambda_b=\sup\{\lambda>0:\mbox{equation}\ \eqref{p}\ \mbox{has non-zero solutions}\}.
\end{equation*}
We alread know from Theorems \ref{nonex} and \ref{existence1} that
\begin{prop}\label{bifubound} There holds
	\begin{equation*}
	\lambda_0^*+\varepsilon\le\lambda_b\le \lambda^*.
	\end{equation*}
\end{prop}
Although we were not able to quantify $\lambda_b$ variationally, we will show examples where $\lambda_b=\lambda^*$ or $\lambda_b<\lambda^*$. To this end we will need the following result
\begin{prop}\label{NON} If $u\in \mathcal{N}_{\lambda^*}$ then
		\begin{equation*}
	pP(u)+\lambda^*\gamma T(u)-qQ(u)=0.
	\end{equation*}
\end{prop}
\begin{proof} In fact, from Proposition \ref{fiberingvariation} and item $iii)$ of Proposition \ref{extremalvalues} it follows that $\mathcal{N}_{\lambda^*}\neq \emptyset$ and if $u\in\mathcal{N}_{\lambda^*}$ then $\lambda(u)=\lambda^*$. Since $u$ is a global maximizer of the $C^1$ function $X\ni v\mapsto \lambda(v)$, it follows that $\lambda'(u)=0$ which implies that
	\begin{equation*}
	\mathcal{P}'(u)+\frac{q-p}{\gamma-q}\frac{\mathcal P(u)}{\mathcal T(u)}\mathcal T'(u)-\frac{\gamma-p}{\gamma-q}\frac{\mathcal P(u)}{\mathcal Q(u)}\mathcal Q'(u)=0.
	\end{equation*}
	Once $u\in \mathcal{N}_{\lambda^*}$ implies that
	\begin{equation*}
	\frac{q-p}{\gamma-q}\frac{\mathcal P(u)}{\mathcal T(u)}=\lambda^*\ \ \ \mbox{and}\ \ \ \frac{\gamma-p}{\gamma-q}\frac{\mathcal P(u)}{\mathcal Q(u)}=1,
	\end{equation*}
	we conclude that
	\begin{equation*}
	pP(u)+\lambda^*\gamma T(u)-qQ(u)=0.
	\end{equation*}
\end{proof}
For the next proposition we assume that $t(u)$ is given by \eqref{tu}. 
\begin{prop}\label{decreconti} For each $u\in X\setminus\{0\}$ there holds
	\begin{enumerate}
		\item[i)] The function $(0,\lambda(u))\ni \lambda\mapsto t_\lambda^+(u)$ is decreasing and continuous.
		\item[ii)] The function $(0,\lambda(u))\ni \lambda\mapsto t_\lambda^-(u)$ is increasing and continuous.
	\end{enumerate} Moreover 
	\begin{equation*}
	\lim_{\lambda\uparrow \lambda(u)}t_\lambda^+(u)=\lim_{\lambda\uparrow \lambda(u)}t_\lambda^-(u)=t(u).
	\end{equation*}
\end{prop}
\begin{proof} Indeed, let $t_\lambda\equiv t_\lambda^+(u)$ and note that $t_\lambda$ satisfies $\psi'_\lambda(t_\lambda)=0$ for each $\lambda\in (0,\lambda(u))$. By implicit differentiation and the fact that $\psi''_\lambda(t_\lambda)>0$, we conclude that $(0,\lambda(u))\ni \lambda\mapsto t_\lambda^+(u)$ is decreasing and continuous, which proves $i)$. The proof of $ii)$ is similar and the limits
	\begin{equation*}
	\lim_{\lambda\uparrow \lambda(u)}t_\lambda^+(u)=\lim_{\lambda\uparrow \lambda(u)}t_\lambda^-(u)=t(u),
	\end{equation*}
	are straightforward from the definitions.
	
\end{proof}
\subsection{A Particular Case} 

In this Subsection we study a particular case of equation \eqref{p} where globally existence of solutions can be proved for all $\lambda\in(0,\lambda^*]$, that is, $\lambda_b=\lambda^*$. In the next Section we give an application of this result to a Kirchhoff type equation. We assume throughout this Subsection that 
\begin{equation*}
\mathcal{T}(u)=\mathcal{P}(u)^{\frac{\gamma}{p}},\ \forall u\in X,
\end{equation*}
and we prove
\begin{theorem}\label{Globalsolu} For each $\lambda\in(0,\lambda^*)$ there exists $u_\lambda,w_\lambda\in X$ which are solutions to equation \eqref{p} and satisies
	\begin{equation*}
	0<\Phi_\lambda(u_\lambda)<\Phi_\lambda(w_\lambda).
	\end{equation*}
	Moreover, there exists at least one solution $v_{\lambda^*}\in X$ to equation \eqref{p} with $\lambda=\lambda^*$ and
	\begin{equation*}
	\lim_{\lambda\uparrow \lambda^*}\Phi_\lambda(u_\lambda)=\lim_{\lambda\uparrow \lambda^*}\Phi_\lambda(w_\lambda)=\Phi_{\lambda^*}(v_{\lambda^*})=\frac{\gamma-p}{pq\gamma}\frac{(q-p)^{\frac{\gamma}{\gamma-p}}}{(\gamma-q)^{\frac{p}{\gamma-p}}}\frac{1}{(\lambda^*)^{\frac{p}{\gamma-p}}.}
	\end{equation*}
	Furthermore
	\begin{equation*}
	\lambda_b=\lambda^*.
	\end{equation*}
\end{theorem}
\begin{rem}\label{contat}
	We note here that Theorem \ref{Globalsolu} and all the results of this Subsection still true if 
	\begin{equation*}
	\mathcal{T}(u)=C_3\mathcal{P}(u)^{\frac{\gamma}{p}},\ \forall u\in X,
	\end{equation*}
	where $C_3$ is a positive constant. Obviously, certain constants that appear have to be adjusted.
\end{rem}
We prove first the existence of the local minimizer. Observe that 
\begin{equation*}
\Phi_\lambda(u)=\frac{1}{p}\mathcal{P}(u)+\frac{\lambda}{\gamma}\mathcal{P}(u)^{\frac{\gamma}{p}}-\frac{1}{q}\mathcal{Q}(u), \forall u\in X,
\end{equation*}
and one can easily see from \ref{E2} that 
\begin{prop}\label{coerP} For each $\lambda>0$ the energy functional $\Phi_\lambda$ is coercive.
\end{prop}
Now we provide some finer estimates over the Nehari sets in order to prove existence of solutions for all $\lambda\in(0,\lambda^*]$.
\begin{prop}\label{n0} For each $\lambda\in(0,\lambda^*]$ for which $\mathcal{N}_\lambda^0\neq\emptyset$, there holds
	\begin{equation*}
	\mathcal{P}(u)=\frac{q-p}{\gamma-q}\frac{1}{\lambda}\ \ \ \mbox{and}\ \ \ \Phi_\lambda(u)=\frac{\gamma-p}{pq\gamma}\frac{(q-p)^{\frac{\gamma}{\gamma-p}}}{(\gamma-q)^{\frac{p}{\gamma-p}}}\frac{1}{\lambda^{\frac{p}{\gamma-p}}},\ \forall u\in \mathcal{N}_\lambda^0.
	\end{equation*}
	Moreover
	\begin{equation*}
		\mathcal{P}(u)<\frac{q-p}{\gamma-q}\frac{1}{\lambda},\ \forall u\in \mathcal{N}_\lambda^-\ \ \ \mbox{and}\ \ \ 	\mathcal{P}(u)>\frac{q-p}{\gamma-q}\frac{1}{\lambda},\ \forall u\in \mathcal{N}_\lambda^+.
	\end{equation*}
\end{prop}
\begin{proof} In fact, if $u\in \mathcal{N}_\lambda^0$, then $\varphi_{\lambda,u}'(1)=\varphi_{\lambda,u}''(1)=0$ which implies that
	\begin{equation}\label{D1}
	\left\{
	\begin{aligned}
	\mathcal{P}(u)+\lambda \mathcal{P}(u)^{\frac{\gamma}{p}}-\mathcal{Q}(u)  &= 0, \\
	p\mathcal{P}(u)+\lambda \gamma\mathcal{P}(u)^{\frac{\gamma}{p}}-q\mathcal{Q}(u)  &= 0.
	\end{aligned}
	\right.
	\end{equation}
	It follows from (\ref{D1}) that
	\begin{equation}\label{D2}
\mathcal{P}(u)=\frac{q-p}{\gamma-q}\frac{1}{\lambda}.
	\end{equation}	
	Moreover, from (\ref{D1}) we also have that 
	\begin{equation}\label{D3}
	\Phi_\lambda(u)=\frac{q-p}{pq}\mathcal{P}(u)-\frac{\gamma-q}{q\gamma}\lambda\mathcal{P}(u)^{\frac{\gamma}{p}}, \forall u\in \mathcal{N}_\lambda^0.
	\end{equation}
	We combine (\ref{D2}) with (\ref{D3}) to get the second equality. The proof of the inequalities are similar, by noting that instead of $\varphi_{\lambda,u}''(1)=0$ in the second line of \eqref{D1}, we would have  $\varphi_{\lambda,u}''(1)<0$ or $\varphi_{\lambda,u}''(1)>0$.
\end{proof}
We see from Proposition \ref{n0} that the energy functional is constant and decreasing over the Nehari set $\mathcal{N}_\lambda^0$. We will use this fact to prove that
\begin{prop}\label{N-comp} Assume that $\lambda\in(0,\lambda^*)$, then for each $u\in \mathcal{N}_\lambda^-$ there holds
	\begin{equation*}
	\Phi_\lambda(u)<\frac{\gamma-p}{pq\gamma}\frac{(q-p)^{\frac{\gamma}{\gamma-p}}}{(\gamma-q)^{\frac{p}{\gamma-p}}}\frac{1}{\lambda^{\frac{p}{\gamma-p}}}.
	\end{equation*}
\end{prop}
\begin{proof} Indeed, suppose on the contrary that there exists $u\in\mathcal{N}_\lambda^-$ such that 
	\begin{equation*}
	\Phi_\lambda(u)\ge\frac{\gamma-p}{pq\gamma}\frac{(q-p)^{\frac{\gamma}{\gamma-p}}}{(\gamma-q)^{\frac{p}{\gamma-p}}}\frac{1}{\lambda^{\frac{p}{\gamma-p}}}.
	\end{equation*}
	From Proposition \ref{decreconti} it follows that $1<t^-_{\lambda'}(u)<t^+_{\lambda'}(u)<t^+_{\lambda}(u)$ for each $\lambda<\lambda'<\lambda(u)$. Therefore
	\begin{eqnarray*}
		\Phi_{\lambda'}(t^-_{\lambda'}(u)u)&>& \Phi_{\lambda'}(u) \nonumber\\       
		&>& \Phi_{\lambda}(u) \nonumber\\
		&\ge& \frac{\gamma-p}{pq\gamma}\frac{(q-p)^{\frac{\gamma}{\gamma-p}}}{(\gamma-q)^{\frac{p}{\gamma-p}}}\frac{1}{\lambda^{\frac{p}{\gamma-p}}}, \nonumber
	\end{eqnarray*}
	which implies that (see Proposition \ref{decreconti} again)
	\begin{equation*}
	\frac{\gamma-p}{pq\gamma}\frac{(q-p)^{\frac{\gamma}{\gamma-p}}}{(\gamma-q)^{\frac{p}{\gamma-p}}}\frac{1}{\lambda^{\frac{p}{\gamma-p}}}<\lim_{\lambda'\uparrow \lambda(u)}	\Phi_{\lambda'}(t^-_{\lambda'}(u)u)=\Phi_{\lambda(u)}(t(u)u)=\frac{\gamma-p}{pq\gamma}\frac{(q-p)^{\frac{\gamma}{\gamma-p}}}{(\gamma-q)^{\frac{p}{\gamma-p}}}\frac{1}{\lambda(u)^{\frac{p}{\gamma-p}}},
	\end{equation*}
	a contradiction since $\lambda<\lambda(u)$.
\end{proof}
As a consequence of Proposition \ref{N-comp}, we now give another proof of Proposition \ref{LocM} (existence of local minimizers) in this particular case, for all $\lambda\in(\lambda_0^*,\lambda^*)$. 
\begin{theorem}\label{LocMPC} For each $\lambda\in (\lambda_0^*,\lambda^*)$ there exists $u_\lambda\in \mathcal{N}_\lambda^+$ such that
	 \begin{equation*}
	\Phi_\lambda(u_\lambda)=\hat{J}_\lambda,\ \ \ \mathcal{P}(u_\lambda)>\frac{q-p}{\gamma-q}\frac{1}{\lambda}\ \ \ \mbox{and}\ \ \ 0<\hat{J}_\lambda<\frac{\gamma-p}{pq\gamma}\frac{(q-p)^{\frac{\gamma}{\gamma-p}}}{(\gamma-q)^{\frac{p}{\gamma-p}}}\frac{1}{\lambda^{\frac{p}{\gamma-p}}}.
	\end{equation*} 
	Moreover $\Phi'_\lambda(u_\lambda)=0$.
\end{theorem}
\begin{proof} Indeed, suppose that $u_n\in \mathcal{N}_\lambda^+\cup\mathcal{N}_\lambda^0$ is a minimizing sequence to $\hat{J}_\lambda^+$. From Proposition \ref{coerP}, we can assume without loss of generality that $u_n\rightharpoonup u$ in $X$. We claim that $u\neq 0$. Indeed if $u=0$, then from the equality
	\begin{equation*}
	\mathcal{P}(u_n)+\lambda \mathcal{P}(u_n)^{\frac{\gamma}{p}}-\mathcal{Q}(u_n)=0, \forall n,
	\end{equation*}	
	and hypothesis \ref{C} and \ref{E2} we obtain that $\|u_n\|\to 0$ as $n\to \infty$, which contradicts Lemma \ref{bound} and therefore $u\neq 0$.
	
	Now we claim that $u_n\to u$ in $X$. If on the contrary we have that $u_n\nrightarrow u$ in $X$, then 
	\begin{equation*}
	\varphi'_{\lambda,u}(1)<\liminf_{n\to \infty}\varphi'_{\lambda,u_n}(1)=0,
	\end{equation*}
	and hence $\varphi_{\lambda,u}$ must satisfies $I)$ of Proposition \ref{fibering} with $t_\lambda^-(u)<1<t_\lambda^+(u)$. Therefore
	\begin{equation*}
	\Phi_\lambda(t_\lambda^+(u)u)<\Phi_\lambda(u)<\liminf_{n\to \infty}\Phi_\lambda(u_n)=\hat{J}_\lambda,
	\end{equation*}
	which is a contradiction since $t_\lambda^+(u)u\in \mathcal{N}_\lambda^+$. We conclude that $u_n\to u$ in $X$ as $n\to \infty$ and consequently $\Phi_\lambda(u)=\hat{J}_\lambda$ and $u\in \mathcal{N}_\lambda^+\cup \mathcal{N}_\lambda^0$, however, from Proposition \ref{N-comp} we have that the energy of $\Phi_\lambda$ ovet $\mathcal{N}_\lambda^0$ is constant and bigger than $\hat{J}_\lambda^+$ and therefore $u\in \mathcal{N}_\lambda^+$,
	\begin{equation*}
	\mathcal{P}(u)>\frac{q-p}{\gamma-q}\frac{1}{\lambda}\ \ \ \mbox{and}\ \ \ \hat{J}_\lambda<\frac{\gamma-p}{pq\gamma}\frac{(q-p)^{\frac{\gamma}{\gamma-p}}}{(\gamma-q)^{\frac{p}{\gamma-p}}}\frac{1}{\lambda^{\frac{p}{\gamma-p}}}.
	\end{equation*}
	 From the definition of $\lambda_0^*$ we conclude that $\hat{J}_\lambda=\Phi_\lambda(u)>0$ and from Lemma \ref{constrained} it follows that $\Phi_\lambda'(u)=0$. By setting $u_\lambda:=u$ the proof is complete.
\end{proof}
Now we turn our attention to the second solution. We start with the following technical Lemma:
\begin{lem}\label{3option} Let $\lambda\in[\lambda_0^*,\lambda^*)$ and assume that $u\in X$ satisfies
	\begin{equation*}
	\mathcal{P}(u)=\frac{q-p}{\gamma-q}\frac{1}{\lambda}.
	\end{equation*}
	If $\varphi_{\lambda,u}$ satisfies $I)$ of Proposition \ref{fibering}, then
	\begin{equation*}
	\Phi_\lambda(u)>\hat{J}_\lambda.
	\end{equation*}
	If $\varphi_{\lambda,u}$ satisfies $II)$ or $III)$ of Proposition \ref{fibering}, then
	\begin{equation*}
	\Phi_\lambda(u)\ge\frac{\gamma-p}{pq\gamma}\frac{(q-p)^{\frac{\gamma}{\gamma-p}}}{(\gamma-q)^{\frac{p}{\gamma-p}}}\frac{1}{\lambda^{\frac{p}{\gamma-p}}}>\hat{J}_\lambda.
	\end{equation*}
\end{lem}
\begin{proof} If $\varphi_{\lambda,u}$ satisfies $I)$ of Proposition \ref{fibering}, then 
	\begin{equation*}
	\Phi_\lambda(u)>\varphi_{\lambda,u}(t_\lambda^+(u))\ge\hat{J}_\lambda,
	\end{equation*}
	where the first inequality is a consequence of Proposition \ref{n0}. In fact, since
	\begin{equation*}
		\mathcal{P}(t^-_\lambda(u)u)<\frac{q-p}{\gamma-q}\frac{1}{\lambda} \ \ \ \mbox{and} \ \ \ \mathcal{P}(t^+_\lambda(u)u)>\frac{q-p}{\gamma-q}\frac{1}{\lambda},
	\end{equation*}
	and $\mathcal{P}(tu)=t^p\mathcal{P}(u)$, we must conclude that $t_\lambda^-(u)<1<t_\lambda^+(u)$.
	
	If $\varphi_{\lambda,u}$ satisfies $II)$ or $III)$ of Proposition \ref{fibering} it follows that $\varphi_{\lambda,u}'(1)>0$, that is
	\begin{equation*}
		\mathcal{P}(u)+\lambda \mathcal{P}(u)^{\frac{\gamma}{p}}-\mathcal{Q}(u)\ge0,
	\end{equation*}
	which implies that 
	\begin{equation*}
		\mathcal{P}(u)+\lambda \mathcal{P}(u)^{\frac{\gamma}{p}}\ge\mathcal{Q}(u).
	\end{equation*}
	If $v\in\mathcal{N}_\lambda^0$, it follows from Proposition \ref{n0} that
	\begin{equation*}
	\mathcal{Q}(v)=\mathcal{P}(v)+\lambda \mathcal{P}(v)^{\frac{\gamma}{p}}=\mathcal{P}(u)+\lambda \mathcal{P}(u)^{\frac{\gamma}{p}}\ge \mathcal{Q}(u),
	\end{equation*}
	and hence
	\begin{equation*}
	\varphi_{\lambda,u}(t)=\frac{1}{p}\mathcal{P}(u)t^p+\frac{\lambda}{\gamma}\mathcal{P}(u)^{\frac{\gamma}{p}}t^\gamma-\frac{1}{q}\mathcal{Q}(u)t^q\ge\frac{1}{p}\mathcal{P}(v)t^p+\frac{\lambda}{\gamma}\mathcal{P}(v)^{\frac{\gamma}{p}}t^\gamma-\frac{1}{q}\mathcal{Q}(v)t^q=\varphi_{\lambda,v}(t),
	\end{equation*}
	for all $t>0$, therefore, from Proposition \ref{N-comp} we conculde that
	\begin{equation*}
	\Phi_\lambda(u)=\varphi_{\lambda,u}(1)\ge\varphi_{\lambda,v}(1)=\frac{\gamma-p}{pq\gamma}\frac{(q-p)^{\frac{\gamma}{\gamma-p}}}{(\gamma-q)^{\frac{p}{\gamma-p}}}\frac{1}{\lambda^{\frac{p}{\gamma-p}}}>\inf\{\Phi_\lambda(w):\ w\in\mathcal{N}_\lambda^-\}\ge \hat{J}_\lambda.
	\end{equation*}
	
\end{proof}
\begin{prop}\label{MPE} For each $\lambda\in[\lambda_0^*,\lambda^*)$ there holds
	\begin{equation*}
	\inf\left\{\Phi_\lambda(u):\mathcal{P}(u)=\frac{q-p}{\gamma-q}\frac{1}{\lambda}\right\}>\hat{J}_\lambda.
	\end{equation*}
\end{prop}
\begin{proof} Indeed, the inequality
	\begin{equation*}
	\inf\left\{\Phi_\lambda(u):\mathcal{P}(u)=\frac{q-p}{\gamma-q}\frac{1}{\lambda}\right\}\ge\hat{J}_\lambda,
	\end{equation*}
	is a consequence of Lemma \ref{3option}. We claim that 
	\begin{equation*}
	\inf\left\{\Phi_\lambda(u):\mathcal{P}(u)=\frac{q-p}{\gamma-q}\frac{1}{\lambda}\right\}>\hat{J}_\lambda.
	\end{equation*}
	In fact, if on the contrary the equality is true, then there exists a sequence 
	\begin{equation*}
	u_n\in X,\ \mathcal{P}(u_n)=\frac{q-p}{\gamma-q}\frac{1}{\lambda},\ n\in \mathbb{N},
	\end{equation*}
	such that $\Phi_\lambda(u_n)\to \hat{J}_\lambda$. From Lemma \ref{3option} and Theorem \ref{LocMPC}, we may assume without loss of generality that $u_n$ satisfies $I)$ of Proposition \ref{fibering} for all $n$ and therefore
	\begin{equation*}
	\Phi_\lambda(u_n)>\Phi_\lambda(t_\lambda^+(u_n)u_n)\ge \hat{J}_\lambda,\ \forall n,
	\end{equation*}
	which implies that $t_\lambda^+(u_n)u_n$ is a minimizing sequence to $\hat{J}_\lambda$. Arguing as in Theorem \ref{LocMPC} we obtain that $t_\lambda^+(u_n)u_n\to t_\lambda^+(u)u$ and $u_n\to u$ as $n\to \infty$ where $\Phi_\lambda( t_\lambda^+(u)u)=\hat{J}_\lambda$. It follows that 
	\begin{equation*}
	\mathcal{P}(u)=\frac{q-p}{\gamma-q}\frac{1}{\lambda}\ \ \ \mbox{and}\ \ \ \Phi_\lambda(u)=\hat{J}_\lambda,
	\end{equation*}
	which clearly contradicts Proposition \ref{n0} and thus
	\begin{equation*}
	\inf\left\{\Phi_\lambda(u):\mathcal{P}(u)=\frac{q-p}{\gamma-q}\frac{1}{\lambda}\right\}>\hat{J}_\lambda.
	\end{equation*}
\end{proof}
As a Corollary of Proposition \ref{MPE} we have a Mountain Pass Geometry for all $\lambda\in(\lambda_0^*,\lambda^*)$: indeed, for all $\lambda\in(\lambda_0^*,\lambda^*)$ fix $u_\lambda$ given by Theorem \ref{LocMPC} and define
\begin{equation*}
\tilde{c}_\lambda=\inf_{\psi\in\tilde{\Gamma}_\lambda}\max_{t\in[0,1]}\Phi_\lambda(\psi(t)),
\end{equation*}
where
\begin{equation*}
\tilde{\Gamma}_\lambda=\{\psi\in C([0,1]:X):\psi(0)=0,\ \psi(1)=u_\lambda\}.
\end{equation*}
\begin{theorem}\label{MPSPCA} For each $\lambda\in(\lambda_0^*,\lambda^*)$ we have that
	\begin{equation*}
		\hat{J}_\lambda<c_\lambda<\frac{\gamma-p}{pq\gamma}\frac{(q-p)^{\frac{\gamma}{\gamma-p}}}{(\gamma-q)^{\frac{p}{\gamma-p}}}\frac{1}{\lambda^{\frac{p}{\gamma-p}}}.
	\end{equation*}
	 Moreover, there exists $w_\lambda\in X$ such that 
	\begin{equation*}
	\Phi_\lambda(w_\lambda)=c_\lambda\ \ \ \mbox{and}\ \ \ \Phi_\lambda'(w_\lambda)=0.
	\end{equation*}
\end{theorem}
\begin{proof} Indeed, from Proposition \ref{MPE} we conclcude that
	\begin{equation*}
c_\lambda>\hat{J}_\lambda=\Phi_\lambda(u_\lambda),
	\end{equation*}
	which implies the desired Mountain Pass Geometry and hence from \ref{PS} we conclude the existence of $w_\lambda\in X$ satisfying
	\begin{equation*}
	\Phi_\lambda(w_\lambda)=c_\lambda\ \ \ \mbox{and}\ \ \ \Phi_\lambda'(w_\lambda)=0.
	\end{equation*}	
	The inequality
	\begin{equation*}
	\frac{\gamma-p}{pq\gamma}\frac{(q-p)^{\frac{\gamma}{\gamma-p}}}{(\gamma-q)^{\frac{p}{\gamma-p}}}\frac{1}{\lambda^{\frac{p}{\gamma-p}}}>c_\lambda,
	\end{equation*}
	can be proven by noting that the path $\psi:[0,1]\to X$ defined by $\psi(t)=\varphi_{\lambda,u_\lambda}(t)$ satisfies
	\begin{equation*}
	\max_{t\in [0,1]}\psi(t)=\varphi_{\lambda,u_\lambda}(t_\lambda^-(u_\lambda))<\frac{\gamma-p}{pq\gamma}\frac{(q-p)^{\frac{\gamma}{\gamma-p}}}{(\gamma-q)^{\frac{p}{\gamma-p}}}\frac{1}{\lambda^{\frac{p}{\gamma-p}}},
	\end{equation*}
	where the last inequality comes from Proposition \ref{N-comp}.
\end{proof}
\begin{rem} For each $\lambda\in(0,\lambda^*)$ define
	\begin{equation*}
	\hat{J}^-_\lambda=\inf\{\Phi_\lambda(u):u\in \mathcal{N}_\lambda^-\cup\mathcal{N}_\lambda^0\}.
	\end{equation*}
Similar to the proof of Theorem \ref{LocMPC} one can prove that there exists $\tilde{w}_\lambda\in \mathcal{N}_\lambda^-$ such that $\Phi_\lambda(\tilde{w}_\lambda)=\hat{J}_\lambda^->0$ and $\Phi_\lambda'(\tilde{w}_\lambda)=0$. Moreover
\begin{equation*}
\tilde{c}_\lambda\ge \hat{J}_\lambda^-.
\end{equation*} 	
\end{rem}
Now we prove the main result of this Subsection:
\begin{proof}[Proof of Theorem \ref{Globalsolu}] The existence of $u_\lambda,w_\lambda\in X$ which are solutions to equation \eqref{p} and satisfies
		\begin{equation*}
		0<\Phi_\lambda(u_\lambda)<\Phi_\lambda(w_\lambda), \forall\lambda\in(\lambda_0^*,\lambda^*),
		\end{equation*}
		is a consequence of Theorems \ref{LocMPC} and \ref{MPSPCA}. The existence of $v_{\lambda^*}$ goes as follows: choose any sequence $\lambda_n\uparrow \lambda^*$ and a corresponding sequence of solutions $u_n:=u_{\lambda_n}$ or $u_n:=w_{\lambda_n}$. Since Theorems \ref{LocMPC} and \ref{MPSPCA} implies that
		\begin{equation*}
		\Phi'_{\lambda_n}(u_n)=0\ \ \ \mbox{and}\ \ \ 0<\Phi_{\lambda_n}(u_n)<	\frac{\gamma-p}{pq\gamma}\frac{(q-p)^{\frac{\gamma}{\gamma-p}}}{(\gamma-q)^{\frac{p}{\gamma-p}}}\frac{1}{(\lambda^*)^{\frac{p}{\gamma-p}}}, \forall n,
		\end{equation*}
		we conclude from \ref{PS} that $u_n\to v$ in $X$ and $\Phi'_{\lambda^*}(v)=0$. From Lemma \ref{bound} we have that $v\neq0$ and since $\mathcal{N}_{\lambda^*}=\mathcal{N}_{\lambda^*}^0$ we conclude that
		\begin{equation*}
		\lim_{\lambda\uparrow \lambda^*}\Phi_\lambda(u_\lambda)=\lim_{\lambda\uparrow \lambda^*}\Phi_\lambda(w_\lambda)=\Phi_{\lambda^*}(v)=\frac{\gamma-p}{pq\gamma}\frac{(q-p)^{\frac{\gamma}{\gamma-p}}}{(\gamma-q)^{\frac{p}{\gamma-p}}}\frac{1}{(\lambda^*)^{\frac{p}{\gamma-p}}}.
		\end{equation*}
The equality $\lambda_b=\lambda^*$ is a consequence of Theorem \ref{nonex}.	By setting $v_{\lambda^*}:=v$ the proof is complete.
\end{proof}
Now we prove Theorem \ref{intglobal}:
\begin{proof}[Proof of Theorem \ref{intglobal}] In fact, the proof follows from Theorems \ref{inexistenceresult} and \ref{Globalsolu}.
	
\end{proof}

\section{Applications}\label{S6}
In this Section we provide some applications. 
\subsection{A Schr\"odinger Equation Coupled With the Electromagnetic Field}\label{SUBSEE}

In the paper of d'Avenia and Siciliano \cite{GaeDa} the following system in $\mathbb R^{3}$ has been studied
\begin{equation}\label{SE}
\left\{
\begin{aligned}
-&\Delta u+\omega u+\lambda\phi u=|u|^{q-2}u, \\
&-\Delta \phi+a^2\Delta^2 \phi = 4\pi u^2,
\end{aligned}
\right.
\end{equation}
where $a> 0$, $\omega> 0$, $\lambda>0$ and $q\in (2,6)$. The system appears when one looks for stationary
solutions $u(x)e^{i\omega t}$ of the Schr\"odinger equation coupled with the Bopp-Podolski 
Lagrangian of the electromagnetic field, in the purely electrostatic situation.
Here $u$ represents the modulus of the wave function and $\phi$ the electrostatic potential.
From a physical point of view, the parameter $\lambda$ has the meaning of the electric charge and $a$ is the parameter of
the Bopp-Podolski term.

In the cited paper, it has been shown that the problem can be addressed variationally. Indeed introducing the Hilbert space  $$\mathcal{D}:=\left\{\phi\in {D}^{1,2}(\mathbb{R}^3):\ \Delta \phi\in L^2(\mathbb{R}^3)\right\}$$
normed by $$\|\phi\|^{2}_{\mathcal D} = a^{2}\|\Delta \phi\|_{2}^{2} + \|\nabla \phi \|_{2}^{2},$$
it can be proved that % the test functions are dense, 
for every $u\in H^{1}(\mathbb R^{3})$ there is a unique solution $\phi_{u}\in \mathcal D$ of the second equation
in the system, that is
\begin{equation}\label{solu1}
-\Delta\phi_u+	a^2\Delta^2\phi_u=4\pi u^2.
\end{equation}
Moreover it turns out that
\begin{equation}\label{Kernel}
\phi_{u} = \frac{1 - e^{-|\cdot|/a}}{|\cdot|} *u^{2}.
\end{equation}

By using the classical by now {\sl reduction argument} one is led to study, equivalently,  the single equation
\begin{equation}\label{eq:equacao}
-\Delta u+\omega u+\lambda\phi_{u} u=|u|^{q-2}u \quad \text{ in } \ \mathbb R^{3}
\end{equation}
containing the  nonlocal term $\phi_{u} u$. In the more recent work of Siciliano and Silva \cite{gaeka}, by considering $q\in (2,3]$, the authors were able to extend the results of \cite{GaeDa} and also the results of Ruiz \cite{Ruiz} (case where a=0). Indeed, note that \eqref{eq:equacao} is a particular case of \eqref{p}, to wit, let
\begin{equation*}
P(u)=-\Delta u+\omega u,\ \ T(u)=\phi_u u,\ \ Q(u)=|u|^{q-2}u,\ u\in X,
\end{equation*}
where $X:=H_r^1(\mathbb{R}^3)\subset H^1(\mathbb{R}^3)$ is the Sobolev space of radial functions. Note here that $p=2$, $\gamma=4$ and $2<q\le 3$. It was show in \cite{gaeka} that $P,T,Q$ satisfies all the hypothesis of this work and hence problem \eqref{SE} has two non-zero solutions on the interval $(0,\lambda_0^*+\varepsilon)$ and does not have non-zero solutions for $\lambda>\lambda^*$. 

 Let us apply Proposition \ref{NON} to this case:
\begin{cor}\label{NON1} If $u\in \mathcal{N}_{\lambda^*}$, then
	\begin{equation*}
	-2\Delta u+2\omega u+4\lambda^*\phi_u u-q|u|^{q-2}u=0.
	\end{equation*}
\end{cor}
We improve the non existence result of \cite{gaeka}:
\begin{prop}\label{NONS}There exists $\epsilon>0$ such that problem \eqref{SE} does not have non-zero solutions for $\lambda\in(\lambda^*-\epsilon,\lambda^*]$.
\end{prop}
\begin{proof} Let us start with the case $\lambda=\lambda^*$. Suppose that $u$ is a solution to problem \eqref{SE}. Since $u\in \mathcal{N}_{\lambda^*}$ it follows from Corollary \ref{NON1} that
	\begin{equation*}
	\left\{
	\begin{aligned}
	&-\Delta u+\omega u+\lambda^*\phi_u u-|u|^{q-2}u=0, \\
	&-2\Delta u+2\omega u+4\lambda^*\phi_u u-q|u|^{q-2}u=0,
	\end{aligned}
	\right.
	\end{equation*}
	and therefore
	\begin{equation*}
	\phi_u(x)u(x)=\frac{q-2}{2\lambda^*}|u(x)|^{q-2}u(x),\ \mbox{a.e}\ x\in \mathbb{R}^3.
	\end{equation*}
	We can assume that $u$ is continuous (see \cite{GaeDa}) and since $\phi_u(x)=0$ if, and only if $u(x)=0$ almost everywhere, it follows that $u$ does not change sign in $\mathbb{R}^3$. We claim that $u(x)=0$ for all $x$. Let us assume, on the contrary and without loss of generality, that $u(x)>0$ for each $x$, then
	\begin{equation}\label{phiu}	
	\phi_u(x)=\frac{q-2}{2\lambda^*}u(x)^{q-2},\ \forall x\in \mathbb{R}^3.
	\end{equation}
	From \cite{GaeDa} we know that $\phi_u\in C^4(\mathbb{R}^3)$ and $\phi_u\in H^4(\mathbb{R}^3)$. Just for the sake of clarity we deal with the case $a=0$, the case $a>0$ being similar.  From equation \eqref{solu1} we obtain that 
	\begin{equation*}
	-\Delta\left(\frac{q-2}{2\lambda^*}u(x)^{q-2}\right)=4\pi u^2,
	\end{equation*}
	which implies that
		\begin{equation*}
	-\frac{q-2}{2\lambda^*}[(q-3)u^{q-4}|\nabla u|^2+u^{q-3}\Delta u]=4\pi u^2,
	\end{equation*}
	and since \eqref{phiu} implies that
	\begin{equation*}
	-\Delta u=-\omega u+\frac{4-q}{2}u^{q-1},
	\end{equation*}
	we conclude that 
		\begin{equation*}
	-\frac{(q-2)^2}{2\lambda^*}\left[(q-3)u^{q-4}|\nabla u|^2+\omega u^{q-4}-\left(\frac{4-q}{2}\right)u^{2q-4}\right]=4\pi u^2,
	\end{equation*}
	and hence
		\begin{equation*}
	-\frac{(q-2)^2}{2\lambda^*}u^{q-6}\left[(q-3)|\nabla u|^2+\omega -\left(\frac{4-q}{2}\right)u^q\right]=4\pi,
	\end{equation*}
	Since $\lim_{x\to \infty}u(x)=0$ and $q-6<0$, we reach a contradiction and therefore $u(x)=0$ for each $x\in\mathbb{R}^3$.
	
	 Now we prove that there exists $\epsilon>0$ such that for each $\lambda\in(\lambda^*-\epsilon,\lambda^*]$ the only solution to equation \eqref{SE} is $u=0$. Indeed, on the contrary we can find a sequence $\lambda_n\uparrow\lambda^*$ and a corresponding sequence of solutions $u_n\neq 0$ to $\eqref{SE}$ with $\lambda=\lambda_n$. Since $u_n$ satisfies \ref{PS} we conclude that $u_n\to u$ in $H^1_r(\mathbb{R}^3)$ and $u$ is a solution to \eqref{SE} for $\lambda=\lambda^*$ and from Lemma \ref{bound} it follows that $u\neq 0$, however, this is a contradiction and hence there exists $\epsilon>0$ such that for each $\lambda\in(\lambda^*-\epsilon,\lambda^*]$ the only solution to equation \eqref{SE} is $u=0$.
\end{proof}
\begin{rem}The difference between the case $a=0$ and $a>0$ is the calculus of $\Delta^2(u^{q-2})$ which is very extensive but results in the same contradiction at the end.
\end{rem}
We conclude this Subsection with the main result of \cite{gaeka}, improved with respect to the non-existence result, which is a consequence of Proposition \ref{NONS}.
\begin{theorem}\label{NSE} There exists $\varepsilon>0$ such that for each $\lambda\in(0,\lambda_0^*+\varepsilon)$ one can find $u_\lambda,w_\lambda\in H_0^1(\Omega)$ which are solutions to equation \eqref{eq:equacao} and satisies
	\begin{equation*}
	\Phi_\lambda(u_\lambda)<0<\Phi_\lambda(w_\lambda),\forall\lambda\in(0,\lambda_0^*),\ \ \ \Phi_{\lambda^*}(u_{\lambda^*})=0,
	\end{equation*}
	and
	\begin{equation*}
	0<\Phi_\lambda(u_\lambda)<\Phi_\lambda(w_\lambda),\forall\lambda\in(\lambda_0^*,\lambda_0^*+\varepsilon).
	\end{equation*}
	Moreover, 
	\begin{equation*}
	\lambda_0^*<\lambda_b<\lambda^*.
	\end{equation*}
\end{theorem}
\subsection{Some Kirchhoff Type Problems}\label{SUBKIr}

The following Kirchhoff type equation was studied in \cite{ka}
\begin{equation}\label{KP}
%\tag{$*$}
\left\{
\begin{aligned}
-\left(a+\lambda\int |\nabla u|^2\right)\Delta u&= |u|^{q-2}u &&\mbox{in}\ \ \Omega, \\
u&=0                                   &&\mbox{on}\ \ \partial\Omega,
\end{aligned}
\right.
\end{equation}
where $a>0$, $\lambda>0$, $q\in (2,4)$ and $\Omega\subset \mathbb{R}^3$ is a bounded regular domain. The main goal of the work was to provide a bifurcation diagram to \eqref{KP} by considering only standard variational techniques, to wit, minimization over the Nehari manifolds and Mountain Pass Theorem.  

Kirchhoff type equations have been extensively studied in the literature. It was proposed by Kirchhoff in \cite{Kir} as a model to study some physical problems related to elastic string vibrations and since then it has been studied by many authors, see for example the works of Lions \cite{Lions}, Alves et al. \cite{clafra}, Wu et al. \cite{wuch}, Zhang and Perera \cite{perzha}, Pucci and R\u{a}dulescu \cite{pura} and the references therein. Physically speaking if one wants to study string or membrane vibrations, one is led to the equation \eqref{p}, where $u$ represents the displacement of the membrane, $|u|^{p-2}u$ is an external force, $a$ and $\lambda$ are related to some intrinsic properties of the membrane. In particular, $\lambda$ is related to the Young modulus of the material and it measures its stiffness.

In \cite{ka} it was show the existence of a positive solution (minimization) for all $\lambda\in(0,\lambda^*]$ and a second positive solution (mountain pass) for $\lambda\in(0,\lambda_0^*+\varepsilon)$ and non existence of solutions for $\lambda>\lambda^*$. We complete this result by giving now a full bifurcation diagram to \eqref{KP}. 
Define
\begin{equation*}
P(u)=-a\Delta u,\ \ T(u)=-\|u\|^2 \Delta u,\ \ Q(u)=|u|^{q-2}u,\ u\in X,
\end{equation*}
where $X=H_0^1(\Omega)$ and $\|u\|$ represents the standard Sobolev norm. It was show in \cite{ka} that $P,T,Q$ satisfies all the hypothesis of this work. Note here that $p=2$, $\gamma=4$ and $2<q<4$. Since
\begin{equation*}
\mathcal{T}(u)=\|u\|^4=\frac{1}{a^2}\mathcal{P}(u)^2, \forall u\in X,
\end{equation*}
it follows from Remark \ref{contat} and Theorem \ref{Globalsolu} that
\begin{theorem}\label{KPFULL} For each $\lambda\in(0,\lambda^*)$ there exists $u_\lambda,w_\lambda\in H_0^1(\Omega)$ which are solutions to equation \eqref{KP} and satisies
	\begin{equation*}
	0<\Phi_\lambda(u_\lambda)<\Phi_\lambda(w_\lambda).
	\end{equation*}
	Moreover, there exists at least one solution $v_{\lambda^*}\in H_0^1(\Omega)$ to equation \eqref{p} with $\lambda=\lambda^*$ and
	\begin{equation*}
	\lim_{\lambda\uparrow \lambda^*}\Phi_\lambda(u_\lambda)=\lim_{\lambda\uparrow \lambda^*}\Phi_\lambda(w_\lambda)=\Phi_{\lambda^*}(v_{\lambda^*})=\frac{(q-2)^2}{4q(4-q)}\frac{a^2}{\lambda^*}.
	\end{equation*}
	Furthermore
	\begin{equation*}
	\lambda_b=\lambda^*.
	\end{equation*}
\end{theorem}
\subsection{A Nonlinear Eigenvalue Problem}\label{SUBNEP} Consider the equation
\begin{equation}\label{NEP}
%\tag{$*$}
\left\{
\begin{aligned}
-\Delta u+\lambda|u|^{\gamma-2}u&= \mu|u|^{q-2}u &&\mbox{in}\ \ \Omega, \\
u&=0                                   &&\mbox{on}\ \ \partial\Omega,
\end{aligned}
\right.
\end{equation}
where $2<q<\gamma<2^*$, $\Omega$ is a bounded regular domain and $\lambda,\mu$ are positive parameters. Define 
\begin{equation*}
P(u)=-\Delta u,\ \ T(u)=|u|^{\gamma-2}u,\ \ Q(u)=\mu|u|^{q-2}u,\ u\in X,
\end{equation*}
for $X=H_0^1(\Omega)$. We claim that $P,Q,T$ satisfies all the hypothesis of this work. Indeed, the hypothesis \ref{H}, \ref{C}, \ref{E1}, \ref{E2} can be easily verified. Hypothesis \ref{E3} is a consequence of
\begin{lem}\label{NEPL1} There exists a constant $C>0$ such that 
	\begin{equation*}
	\frac{\mathcal{Q}(u)^{\frac{\gamma-2}{q-2}}}{\mathcal{T}(u)\mathcal{P}(u)^{\frac{\gamma-q}{q-2}}}\le C, \forall u\in X\setminus\{0\}.
	\end{equation*}
\end{lem}
\begin{proof} In fact, note that
	\begin{equation*}
	\frac{\mathcal{Q}(u)^{\frac{\gamma-2}{q-2}}}{\mathcal{T}(u)\mathcal{P}(u)^{\frac{\gamma-q}{q-2}}}=\frac{\left(\mu\int |u|^q\right)^{\frac{\gamma-2}{q-2}}}{\int |u|^\gamma\left(\int |\nabla u|^2\right)^{\frac{\gamma-q}{q-2}}},\forall u\in H_0^1(\Omega)\setminus\{0\}.
	\end{equation*}
Once $H_0^1(\Omega)\hookrightarrow L^\gamma(\Omega)\hookrightarrow L^q(\Omega)$	the proof is complete.
\end{proof}
Now we prove the \ref{PS} condition
\begin{lem}\label{NEPL2} The energy functional $\Phi_\lambda$ satisfies condition \ref{PS}.
\end{lem}
\begin{proof} Indeed, assume that $\lambda_n\to\lambda> 0$ and $u_n\in H_0^1(\Omega)$ satisfies $\Phi_\lambda(u_n)$ is bounded and $\Phi_\lambda'(u_n)\to 0$ as $n\to \infty$. We claim that $u_n$ is bounded in $H_0^1(\Omega)$. In fact, if not then up to a subsequence, $\|u_n\|\to \infty$ as $n\to \infty$ which implies that ($\Phi_\lambda$ is coercive) 
	\begin{equation*}
	\lim_{n\to \infty}\Phi_\lambda(u_n)\ge \lim_{n\to \infty}\left(\frac{1}{2}\int |\nabla u_n|^2+\frac{\lambda_n}{\gamma}\int |u_n|^\gamma-C\frac{\mu}{q}\left(\int |u_n|^\gamma\right)^{\frac{q}{\gamma}}\right)=\infty,
	\end{equation*}
	which is a contradiction since $\Phi_\lambda(u_n)$ is bounded. Therefore $\|u_n\|$ is bounded and up to a subsequence we can assume that $u_n\rightharpoonup u$ in $H_0^1(\Omega)$ and $u_n\to u$ in $L^q(\Omega),L^\gamma(\Omega)$ as $n\to \infty$. It follows that
	\begin{equation*}
	-\Delta u_n(u_n-u)=-\lambda_n|u_n|^{\gamma-2}u_n(u_n-u)+\mu|u_n|^{q-2}u_n(u_n-u)+o(1)=o(1),
	\end{equation*}
	which implies that $u_n\to u$ in $H_0^1(\Omega)$.
\end{proof}
Since \eqref{NEP} also depends on $\mu$, we write $\lambda_0^*(\mu),\lambda^*(\mu),\lambda_b(\mu),\varepsilon(\mu)$, where $\varepsilon(\mu)$ is given by Proposition \ref{enerynear} for each $\mu>0$ fixed
\begin{theorem}\label{THNEP} For each $\mu>0$, there exists $\varepsilon>0$ such that for each $\lambda\in(0,\lambda_0^*(\mu)+\varepsilon(\mu))$ one can find $u_\lambda,w_\lambda\in H_0^1(\Omega)$ which are solutions to equation \eqref{KP} and satisies
		\begin{equation*}
	\Phi_\lambda(u_\lambda)<0<\Phi_\lambda(w_\lambda),\forall\lambda\in(0,\lambda_0^*(\mu)),\ \ \ \Phi_{\lambda^*(\mu)}(u_{\lambda^*(\mu)})=0,
	\end{equation*}
	and
	\begin{equation*}
	0<\Phi_\lambda(u_\lambda)<\Phi_\lambda(w_\lambda),\forall\lambda\in(\lambda_0^*(\mu),\lambda_0^*(\mu)+\varepsilon(\mu)).
	\end{equation*}
	Moreover, 
	\begin{equation*}
	\lambda_0^*(\mu)<\lambda_b(\mu)<\lambda^*(\mu).
	\end{equation*}
\end{theorem}
\begin{proof} Theorems \ref{existence} and \ref{existence1} guarantee the existence of $u_\lambda,w_\lambda$ satisfying those inequalities. We alread know that $\lambda_b(\mu)\le \lambda^*(\mu)$. Let us prove that $\lambda_b(\mu)<\lambda^*(\mu)$. Indeed, first observe from Proposition \ref{NON} that if $u$ satisfies equation \eqref{NEP} with $\lambda=\lambda^*(\mu)$, then  
\begin{equation*}
\left\{
\begin{aligned}
&-\Delta u+\lambda^*(\mu)|u|^{\gamma-2}u-|u|^{q-2}u=0, \\
&-2\Delta u+\gamma\lambda^*(\mu)|u|^{\gamma-2}u-q|u|^{q-2}u=0,
\end{aligned}
\right.
\end{equation*}
and therefore
\begin{equation*}
(\gamma-2)\lambda^*(\mu)|u|^{\gamma-2}u=(q-2)|u|^{q-2}u,\ \mbox{a.e.}\ x\in \Omega,
\end{equation*}	
which implies that $u=0$. It follows that equation \eqref{NEP} has no non-zero solutions when $\lambda=\lambda^*(\mu)$. Arguing as in the proof of Proposition \ref{NONS}, we conclude that there exists $\epsilon>0$ such that for each $\lambda\in(\lambda^*(\mu)-\epsilon,\lambda^*(\mu)]$ the only solution to equation \eqref{NEP} is $u=0$ and therefore $\lambda_b(\mu)<\lambda^*(\mu)$.
\end{proof}
Before we prove our next result, let us turn our attention to the functions $(0,\infty)\ni\mu\mapsto \lambda_0^*(\mu),\lambda^*(\mu),\lambda_b(\mu),\varepsilon(\mu)$ and prove some auxiliary Lemmas. Define
\begin{equation*}
\mathcal{M}=\sup\left\{\frac{\left(\int |u|^q\right)^{\frac{\gamma-2}{q-2}}}{\int |u|^\gamma\left(\int |\nabla u|^2\right)^{\frac{\gamma-q}{q-2}}},\forall u\in H_0^1(\Omega)\setminus\{0\}\right\}.
\end{equation*}
\begin{lem}\label{auxile} For each $\mu>0$ there holds
	\begin{equation*}
	\lambda_0^*(\mu)=\frac{\gamma}{q}\frac{q-2}{\gamma-2}\left(\frac{2}{q}\frac{\gamma-q}{\gamma-2}\right)^{\frac{\gamma-q}{q-2}}\mathcal{M}\mu^{\frac{\gamma-2}{q-2}},
	\end{equation*}
	and
	\begin{equation*}
	\lambda^*(\mu)=\frac{q-2}{\gamma-2}\left(\frac{\gamma-q}{\gamma-2}\right)^{\frac{\gamma-q}{q-2}}\mathcal{M}\mu^{\frac{\gamma-2}{q-2}},
	\end{equation*}
\end{lem}
\begin{proof} Once 
		\begin{equation*}
	\frac{\mathcal{Q}(u)^{\frac{\gamma-2}{q-2}}}{\mathcal{T}(u)\mathcal{P}(u)^{\frac{\gamma-q}{q-2}}}=\frac{\left(\mu\int |u|^q\right)^{\frac{\gamma-2}{q-2}}}{\int |u|^\gamma\left(\int |\nabla u|^2\right)^{\frac{\gamma-q}{q-2}}},\forall u\in H_0^1(\Omega)\setminus\{0\},
	\end{equation*}
	the proof is a consequence of the definitions of $\lambda_0^*(\mu)$ and $\lambda^*(\mu)$.
\end{proof}
The next Lemma, which is similar to Proposition \ref{enerynear}, however, it takes into account the parameter $\mu$, can be proved by using the uniformly continuity of $\lambda^*(\mu),\lambda_0^*(\mu)$ on compact intervals $[a,b]$ and the fact that the set
\begin{equation*}
\{u_{\lambda_0^*(\mu)}:\mu\in[a,b]\}
\end{equation*}
is compact.
\begin{lem}\label{epsilonv} Suppose that $[a,b]\subset (0,\infty)$. Given $\delta>0$, there exists $\theta>0$ such that for each $\lambda\in(\lambda_0^*(\mu),\lambda_0^*(\mu)+\theta)$ and $\mu\in[a,b]$ there holds $\hat{J}_\lambda<\delta$.
\end{lem}
In Rabinowitz \cite{rabino} Theorem 2.32 (see also Ambrosetti and Rabinowitz \cite{rabinoam}) the following result was proved
\begin{theorem}\label{rabino} Suppose that $\mu=\lambda$, then there exists $\underline{\lambda}>0$ such that for each $\lambda>\underline{\lambda}$, problem \eqref{NEP} has two positive solutions $u_\lambda,w_\lambda\in H_0^1(\Omega)$ satisfying 
	\begin{equation*}
	\Phi_\lambda(u_\lambda)<0<\Phi_\lambda(w_\lambda).
	\end{equation*}
\end{theorem}
\begin{rem} As one can see in \cite{rabino}, $u_\lambda,w_\lambda$ were found as critical points of a modified energy functional. In fact $u_\lambda$ is a global minimum and $w_\lambda$ is a mountain pass critical point. 
\end{rem}
We now show that Theorem \ref{rabino} can be proved by using the technique proposed here. In fact, we prove 
\begin{theorem}\label{rabinoimproved} There exists $\varepsilon>0$ and $\lambda_*>0$ such that problem \eqref{NEP} admits two positive classical solutions $u_\lambda,w_\lambda$ for each $\lambda>\lambda_*-\varepsilon$ and $\mu=\lambda$ satisfying:
	\begin{enumerate}
		\item If $\lambda\ge \lambda_*$, then $u_\lambda$ is a global minimizer to $\Phi_\lambda$ while $w_\lambda$ is a mountain pass solution and
		\begin{equation*}
		\Phi_\lambda(u_\lambda)<0<\Phi_\lambda(w_\lambda), \forall\ \lambda\in[\lambda_*,\infty),
		\end{equation*}
		\begin{equation*}
		\Phi_{\lambda_*}(u_{\lambda_*})=0<\Phi_{\lambda_*}(w_{\lambda_*}).
		\end{equation*}
		\item If $\lambda_*-\varepsilon<\lambda<\lambda_*$, then $u_\lambda$ is a local minimizer to $\Phi_\lambda$ while $w_\lambda$ is a mountain pass solution and
		\begin{equation*}
		0<\Phi_\lambda(u_\lambda)<\Phi_\lambda(w_\lambda), \forall\ \lambda\in(\lambda_*-\varepsilon,\lambda_*).
		\end{equation*}
	\end{enumerate} 
 Moreover, if $\mu_0>0$ is the unique value for which 
	\begin{equation*}
	\mu_0=\lambda^*(\mu_0),
	\end{equation*}
	then problem \eqref{NEP} does not have non-zero solutions for $0<\lambda<\mu_0$ and $\mu=\lambda$.
\end{theorem}
\begin{proof} Indeed, from Lemma \ref{auxile} we have that 
\begin{equation*}
\lambda_0^*(\mu)=\frac{\gamma}{q}\frac{q-2}{\gamma-2}\left(\frac{2}{q}\frac{\gamma-q}{\gamma-2}\right)^{\frac{\gamma-q}{q-2}}\mathcal{M}\mu^{\frac{\gamma-2}{q-2}},\forall\mu>0,
\end{equation*}	
and
\begin{equation*}
\lambda^*(\mu)=\frac{q-2}{\gamma-2}\left(\frac{\gamma-q}{\gamma-2}\right)^{\frac{\gamma-q}{q-2}}\mathcal{M}\mu^{\frac{\gamma-2}{q-2}},\forall\mu>0.
\end{equation*}
Since 
\begin{equation*}
\frac{\gamma-2}{q-2}>1,
\end{equation*}
and $\lambda_0^*(\mu)<\lambda^*(\mu)$ for all $\mu>0$, there exists $0<\mu_0<\lambda_*$ such that
\begin{equation}\label{nonexistence1}
\mu>\lambda^*(\mu),\forall\mu\in(0,\mu_0);\ \ \ \mu_0=\lambda^*(\mu_0);\ \ \ \mu<\lambda^*(\mu),\forall\mu>\mu_0,
\end{equation}
and
\begin{equation}\label{existenceee}
\mu>\lambda_0^*(\mu),\forall\mu\in(0,\lambda_*);\ \ \ \lambda_*=\lambda_0^*(\lambda_*);\ \ \ \mu<\lambda_0^*(\mu),\forall\mu>\lambda_*.
\end{equation}
From now on we suppose that $\mu=\lambda$.

By one hand, the non-existence result for $0<\lambda<\mu_0$ is a consequence of \eqref{nonexistence1} and Theorem \ref{nonex}.

On the other hand, if $\lambda\ge\lambda_*$, then \eqref{existenceee} combined with Theorem \ref{THNEP} implies the existence of $u_\lambda,w_\lambda\in X\setminus\{0\}$ satisfying: $u_\lambda$ is a global minimizer while $w_\lambda$ is a mountain pass solution and
\begin{equation*}
\Phi_\lambda(u_\lambda)<0<\Phi_\lambda(w_\lambda), \forall \lambda\ge \lambda_*,
\end{equation*}
\begin{equation*}
\Phi_{\lambda_0^*}(u_{\lambda_0^*})=0<\Phi_{\lambda_0^*}(w_{\lambda_0^*}).
\end{equation*}
To conclude the proof, fix a interval of the form $[a,\lambda_*]$ with $a\in (\mu_0,\lambda_*)$. Observe that the constants defined after Lemma \ref{boundN0} and before Proposition \ref{EKE1}, which are used to prove existence of two solutions after $\lambda_0^*(\mu)$, can all be choose uniformly with respect to $\mu\in[a,\lambda_*]$. Therefore, we can assume that $\varepsilon(\mu)$ is choosen in such a way that $\varepsilon(\mu)>\theta$ for all $\mu\in[a,\lambda_*]$, where $\theta>0$ is given by Lemma \ref{epsilonv}. 

It follows that there exists $\varepsilon>0$ such that for all $\mu\in(\lambda_*-\varepsilon,\lambda_*)$ we have that
\begin{equation*}
\lambda_0^*(\mu)<\mu<\lambda_0^*(\mu)+\theta<\lambda_0^*(\mu)+\varepsilon(\mu),\forall\mu\in(\lambda_*-\varepsilon,\lambda_*),
\end{equation*}
and hence Theorem \ref{THNEP} implies the existence of $u_\lambda,w_\lambda\in X\setminus\{0\}$ satisfying: $u_\lambda$ is a local minimizer while $w_\lambda$ is a mountain pass solution and
\begin{equation*}
0<\Phi_\lambda(u_\lambda)<\Phi_\lambda(w_\lambda), \forall \lambda \in(\lambda_*-\varepsilon,\lambda_*).
\end{equation*}
\end{proof}
\subsection{A Chern-Simons-Schr\"odinger System}\label{SUBCSS}
In this last application we study the following gauged Schr\"odinger equation in dimension $2$ including the so-called Chern-Simons term:
\begin{equation}\label{CSS}
%\tag{$*$}
\left\{
\begin{aligned}
-\Delta u+u+\lambda\left(\frac{h^2(|x|)}{|x|^2}+\int_{|x|}^\infty\frac{h(s)}{s}u^2(s)ds\right)u&= |u|^{q-2}u &&\mbox{in}\ \ \mathbb{R}^2, \\
u\in H_r^1(\mathbb{R}^2),&                                 
\end{aligned}
\right.
\end{equation}
where $\lambda>0$ is a real positive parameter, $q\in (2,4)$ and 
\begin{equation*}
h(s)=\frac{1}{2}\int _0^s ru^2(r)dr.
\end{equation*}
Equation \eqref{CSS} comes from the study of the standing waves of Chern-Simons-Schr\"odinger System which describes the dynamics of a large number of particles in a electromagnetic field. For a more detailed account of the physical interpretation of equation \eqref{CSS} and previous results we refer the reader to the works of Pomponio and Ruiz \cite{pomru} and the recent work of Xia \cite{xia}. Let
\begin{equation*}
P(u)=-\Delta u+ u,\ \ T(u)=\left(\frac{h^2(|x|)}{|x|^2}+\int_{|x|}^\infty\frac{h(s)}{s}u^2(s)ds\right)u,\ \ Q(u)=|u|^{q-2}u,\ u\in X,
\end{equation*}
where $X:=H_r^1(\mathbb{R}^2)\subset H^1(\mathbb{R}^2)$ is the Sobolev space of radial functions. Note here that $p=2$, $\gamma=6$ and $2<q< 4$. It was show in \cite{xia} that $P,T,Q$ satisfies all the hypothesis of this work and hence 
\begin{theorem}[Xia \cite{xia}] \label{NSE} There exists $\varepsilon>0$ such that for each $\lambda\in(0,\lambda_0^*+\varepsilon)$ one can find $u_\lambda,w_\lambda\in H_r^1(\mathbb{R}^2)$ which are solutions to equation \eqref{CSS} and satisies
	\begin{equation*}
	\Phi_\lambda(u_\lambda)<0<\Phi_\lambda(w_\lambda),\forall\lambda\in(0,\lambda_0^*),\ \ \ \Phi_{\lambda^*}(u_{\lambda^*})=0,
	\end{equation*}
	and
	\begin{equation*}
	0<\Phi_\lambda(u_\lambda)<\Phi_\lambda(w_\lambda),\forall\lambda\in(\lambda_0^*,\lambda_0^*+\varepsilon).
	\end{equation*}
	Moreover, 
	\begin{equation*}
	\lambda_0^*<\lambda_b\le \lambda^*.
	\end{equation*}
\end{theorem}
\begin{rem} We were not able to study $\lambda_b$ for this equation.
\end{rem}
%%%%%%%%%%%%%%%%%%%%%%%%%%%%%%%%%%%%%%%%


\bibliographystyle{amsplain}
\bibliography{Ref}

\providecommand{\bysame}{\leavevmode\hbox to3em{\hrulefill}\thinspace}
\providecommand{\MR}{\relax\ifhmode\unskip\space\fi MR }
\providecommand{\MRhref}[2]{%
  \href{http://www.ams.org/mathscinet-getitem?mr=#1}{#2}
}
\providecommand{\href}[2]{#2}
\begin{thebibliography}{10}

\bibitem{clafra}
C.~O. Alves, F.~J. S.~A. Corr\^ea, and T.~F. Ma, \emph{Positive solutions for a
  quasilinear elliptic equation of {K}irchhoff type}, Comput. Math. Appl.
  \textbf{49} (2005), no.~1, 85--93. \MR{2123187}

\bibitem{rabinoam}
Antonio Ambrosetti and Paul~H. Rabinowitz, \emph{Dual variational methods in
  critical point theory and applications}, J. Functional Analysis \textbf{14}
  (1973), 349--381. \MR{0370183}

\bibitem{browu}
Kenneth~J. Brown and Tsung-Fang Wu, \emph{A fibering map approach to a
  potential operator equation and its applications}, Differential Integral
  Equations \textbf{22} (2009), no.~11-12, 1097--1114. \MR{2555638}

\bibitem{chab}
Jan Chabrowski, \emph{Variational methods for potential operator equations}, De
  Gruyter Studies in Mathematics, vol.~24, Walter de Gruyter \& Co., Berlin,
  1997, With applications to nonlinear elliptic equations. \MR{1467724}

\bibitem{wuch}
Ching-yu Chen, Yueh-cheng Kuo, and Tsung-fang Wu, \emph{The {N}ehari manifold
  for a {K}irchhoff type problem involving sign-changing weight functions}, J.
  Differential Equations \textbf{250} (2011), no.~4, 1876--1908. \MR{2763559}

\bibitem{crarabino}
Michael~G. Crandall and Paul~H. Rabinowitz, \emph{Bifurcation from simple
  eigenvalues}, J. Functional Analysis \textbf{8} (1971), 321--340.
  \MR{0288640}

\bibitem{GaeDa}
P.~d'Avenia and G.~Siciliano, \emph{Nonlinear {S}chr\"odinger equation in the
  {B}opp-{P}odolsky electrodynamics: Solutions in the electrostatic case}, J.
  Differential Equations, https://doi.org/10.1016/j.jde.2019.02.001.

\bibitem{ilyasENMM}
Yavdat Ilyasov, \emph{On extreme values of {N}ehari manifold method via
  nonlinear {R}ayleigh's quotient}, Topol. Methods Nonlinear Anal. \textbf{49}
  (2017), no.~2, 683--714. \MR{3670482}

\bibitem{Kir}
G~Kirchhoff, \emph{Mechanik}, Teubner, Leipzig (1883).

\bibitem{Lions}
J.-L. Lions, \emph{On some questions in boundary value problems of mathematical
  physics}, Contemporary developments in continuum mechanics and partial
  differential equations ({P}roc. {I}nternat. {S}ympos., {I}nst. {M}at.,
  {U}niv. {F}ed. {R}io de {J}aneiro, {R}io de {J}aneiro, 1977), North-Holland
  Math. Stud., vol.~30, North-Holland, Amsterdam-New York, 1978, pp.~284--346.
  \MR{519648}

\bibitem{neh1}
Zeev Nehari, \emph{On a class of nonlinear second-order differential
  equations}, Trans. Amer. Math. Soc. \textbf{95} (1960), 101--123.
  \MR{0111898}

\bibitem{neh}
\bysame, \emph{Characteristic values associated with a class of non-linear
  second-order differential equations}, Acta Math. \textbf{105} (1961),
  141--175. \MR{0123775}

\bibitem{poh}
S.~I. Pokhozhaev, \emph{The fibration method for solving nonlinear boundary
  value problems}, Trudy Mat. Inst. Steklov. \textbf{192} (1990), 146--163,
  Translated in Proc. Steklov Inst. Math. {{\bf{1}}992}, no. 3, 157--173,
  Differential equations and function spaces (Russian). \MR{1097896}

\bibitem{pomru}
Alessio Pomponio and David Ruiz, \emph{A variational analysis of a gauged
  nonlinear {S}chr\"{o}dinger equation}, J. Eur. Math. Soc. (JEMS) \textbf{17}
  (2015), no.~6, 1463--1486. \MR{3353806}

\bibitem{pura}
P~Pucci and V.~D R\u{a}dulescu, \emph{Progress in nonlinear {K}irchhoff
  problems}, Nonlinear Analysis (2019),
  https://doi.org/10.1016/j.na.2019.02.022.

\bibitem{rabino}
Paul~H. Rabinowitz, \emph{Minimax methods in critical point theory with
  applications to differential equations}, CBMS Regional Conference Series in
  Mathematics, vol.~65, Published for the Conference Board of the Mathematical
  Sciences, Washington, DC; by the American Mathematical Society, Providence,
  RI, 1986. \MR{845785}

\bibitem{Ruiz}
David Ruiz, \emph{The {S}chr\"odinger-{P}oisson equation under the effect of a
  nonlinear local term}, J. Funct. Anal. \textbf{237} (2006), no.~2, 655--674.
  \MR{2230354}

\bibitem{gaeka}
Gaetano Siciliano and Kaye Silva, \emph{The fibering method approach for a
  non-linear {S}chr\"odinger equation coupled with the electromagnetic field},
  Publicacions Matem\`atiques (To Appear).

\bibitem{ka}
Kaye Silva, \emph{The bifurcation diagram of an elliptic kirchhoff-type
  equations with respect to the stiffness of the material}, Z. Angew. Math.
  Phys. \textbf{70} (2019), no.~93.

\bibitem{xia}
A~Xia, \emph{Existence, nonexistence and multiplicity results of a
  {C}hern-{S}imons-{S}chr\"{o}dinger system}, Acta Appl Math (2019),
  https://doi.org/10.1007/s10440--019--00260--6.

\bibitem{perzha}
Zhitao Zhang and Kanishka Perera, \emph{Sign changing solutions of {K}irchhoff
  type problems via invariant sets of descent flow}, J. Math. Anal. Appl.
  \textbf{317} (2006), no.~2, 456--463. \MR{2208932}

\end{thebibliography}

\end{document}